\documentclass{article}
\usepackage{hyperref}
\usepackage[american]{babel}
\usepackage{amsfonts,amsmath,amssymb,epsf,epsfig,bbold}

\usepackage{amsthm}

\usepackage{xypic}
\xyoption{all}
\input{xypic}
\newdir{ >}{{}*!/-8pt/\dir{>}}

\renewcommand{\hat}{\widehat} 
\newcommand{\im}{\mathop{{\rm im}}\nolimits}
\newcommand{\Hom}{\mathop{{\rm Hom}}\nolimits}

\def\R{\mathbb{R}}\def\V{\mathbb{V}}
\def\W{\mathbb{W}}

\def\ker{{\rm ker}}
\def\im{{\rm im}}

\def\U{\mathbb{U}}

\def\K{\mathbb{K}}
      \def\Hom{{\rm Hom}}
  \def\Sq{{\rm Sq}} \def\End{{\rm End}}

\newtheorem{theo}{Theorem}[section]
\newtheorem{Defi}[theo]{{\bf Definition}}
\newenvironment{defi}{\begin{Defi} \normalfont}{\end{Defi}}

\newtheorem{Prop}[theo]{{\bf Proposition}}
\newenvironment{prop}{\begin{Prop} \normalfont}{\end{Prop}}
\newtheorem{Cor}[theo]{{\bf Corollary}}
\newenvironment{cor}{\begin{Cor} \normalfont}{\end{Cor}}
\newtheorem{Lem}[theo]{{\bf Lemma}}
\newenvironment{lem}{\begin{Lem} \normalfont}{\end{Lem}}
\newtheorem{pro}[theo]{Problem}
\newtheorem{Exa}[theo]{{\bf Example}}
\newenvironment{exa}{\begin{Exa} \normalfont}{\end{Exa}}
\newtheorem{Exem}{{\bf Example}}[section]

\newtheorem{Rem}[theo]{{\bf Remark}}
\newenvironment{rem}{\begin{Rem} \normalfont}{\end{Rem}}

\newenvironment{prf}{\begin{proof}}{\end{proof}}

\begin{document}
\title{Enhanced Leibniz Algebras: \\Structure Theorem and Induced Lie 2-Algebra} 

\author{Thomas Strobl\\
Universit\'e Lyon 1
\thanks{strobl@math.univ-lyon1.fr}
\and Friedrich Wagemann\\
     Universit\'{e} de Nantes
     \thanks{wagemann@math.univ-nantes.fr}}   
\maketitle

\begin{abstract}
An enhanced Leibniz algebra is an algebraic struture that arises in the context of particular higher gauge theories describing self-interacting gerbes. It consists of a Leibniz algebra $(\V,[ \cdot, \cdot ])$, a bilinear form on $\V$ with values in another vector space $\W$, and a map $t \colon \W \to \V$, satisfying altogether four compatibility relations. Our structure theorem asserts that an enhanced Leibniz algebra is uniquely determined by the underlying Leibniz algebra $(\V,[ \cdot, \cdot ])$, an appropriate abelian ideal ${\mathfrak i}$ inside it, as well as a cohomology 2-class $[\Delta]$ which only effects the $\W$-valued product. 

Positive quadratic enhanced Leibniz algebras, as needed for the definition of a Yang-Mills type action functional, turn out to be rather restrictive on the  underlying Leibniz algebra $(\V,[ \cdot, \dot ])$: $\V$ has to be the hemisemidirect product of a positive quadratic Lie algebra  ${\mathfrak g}$  with a   ${\mathfrak g}$-module ${\mathfrak i}$, $\V \cong {\mathfrak g}\ltimes{\mathfrak i}$, with  ${\mathfrak i}$ the above-mentioned ideal in this case.

The second main result of this article is the construction of a functor from the category of such enhanced Leibniz algebras to the category of (semi-strict) Lie 2-algebras  or, equivalentely, of  two-term $L_{\infty}$-algebras. 
\end{abstract}

\section*{Introduction---Motivation}

Enhanced Leibniz algebras (ELAs) are a particular algebraic structure that arose in the context of non-abelian gauge theories for 1- and 2-form gauge fields, i.e.\ in the context of non-abelian gerbes \cite{Str}. While every Leibniz algebra $(\V, [ \cdot , \cdot ])$---a magma where the left multiplication is a derivation, $[v_1,[v_2,v_3]]=[[v_1,v_2],v_3] + [v_2,[v_1,v_3]]$---gives rise to such a structure, ELAs are more general. To find out how much more general they are is part of the goal for this article. Another main interest into ELAs arises from the fact that they give rise to particular Lie 2-algebras. This last aspect is in fact again closely related to the origin of this new notion in the physics of gauge theories, which therefore seems adequate to recall here as a motivation for both, the notion of an ELA itself as well as its relation to the Lie 2-algebra (equivalently, 2-term $L_\infty$-algebra).

Standard non-abelian gauge theories, as they are known to describe elementary interaction fields, are governed by the Yang-Mills functional. This is a functional defined on the connections of a principal bundle $P$ over the spacetime manifold $M$. For its definition, one needs a quadratic Lie algebra $({\mathbb{g}},[ \cdot , \cdot ],\kappa_\mathbb{g})$. In the case of a trivial bundle, the connections are parametrized by ${\mathbb{g}}$-valued 1-forms, $A\in \Omega^1(M,\mathbb{g})$, and the functional takes the  form 
\begin{equation} S_{YM}[A]= \int_M \kappa_\mathbb{g}(F_A \stackrel{\wedge}{,} \star F_A) \, .
\label{YM}
\end{equation}
Here $F_A = \mathrm{d}_M A + \tfrac{1}{2} [A\stackrel{\wedge}{,}A]$ is the curvature 2-form of the connection and $\star$ denotes the Hodge duality operation defined on $M$. Gauge symmetries are invariances of the functional \eqref{YM} corresponding to vertical automorphisms of the principal bundle under which the curvatures $F_A$ transform in the adjoint representation. The physics of such theories imply in addition that the invariant symmetric bilinear form  is even positive, i.e.\ $\kappa_\mathbb{g}$ defines an invariant scalar product on $\mathbb{g}$.

Electromagnetism is described by the above for the case that $P$ is a $U(1)$-bundle. Here locally $A$ is just an $\R$-valued 1-form on $M$. A gerbe is the analog of this, but defined for an $\R$-valued 2-form $\beta$ on $M$. In the case of several 1-form gauge fields $A$, self-interactions arise from the terms quadratic in $A$ in the 2-form curvature $F_A$, where consistency of the theory leads one have those governed by a Lie algebra. The analog of this for several 2-form gauge fields $\beta$ is not evident, since the curvature of a 2-form gauge field $\beta$ has to be a 3-form, it starts with $\mathrm{d} B$ and there are no terms non-linear in $\beta$ that one can add to this for form-degree reasons. 

This changes if one permits self-interacting gerbes to be described by 2-forms $\beta$ together with 1-forms $A$, $A \in \Omega^1(M,\V)$ and $B \in \Omega^2(M,\W)$, where $\V$ and $\W$ are, at this point, arbitrary vector spaces used only to label these gauge fields. From a physical perspective, one wants that these gauge fields are described by a functional similar to the one of Yang and Mills above. Now there are two possible perspectives one can adopt: 
\begin{enumerate}\item[{\bf 1.}] The ''3-form curvatures'' $G \in \Omega^3(M,\W)$ should enter such a functional quadratically, but the 2-forms $F \in \Omega^2(M,\V)$ should be constrained to vanish and thus enter the corresponding non-abelian gerbe functional linearly with a Lagrange multiplier in front. 
\item[{\bf 2.}] Both, $G$ and $F$ should enter the functional quadratically, following the model of \eqref{YM}:
\begin{equation} S_{gerbes}[A,B]= \int_M \kappa_\mathbb{\V}(F \stackrel{\wedge}{,} \star F) + \kappa_\mathbb{\W}(G \stackrel{\wedge}{,} \star G) \, .
\label{gerbe}
\end{equation}
\end{enumerate}

It is a general, undisputed belief in the physics community that every gauge theory for coupled 1- and 2-form gauge fields should be governed by \emph{some} Lie 2-algebra.\footnote{See, for example, \cite{Baez02}---and \cite{GruStr} for a possible systematic argument.} A strict Lie 2-algebra can be described by a crossed module of Lie algebras \cite{Baez02}, which in turn can be provided by the following data: 
a Lie algebra  $\mathbb{g}$, a $\mathbb{g}$-module $\W$, and an equivariant linear map $t \colon \W \to \mathbb{g}$ such that  for all $w \in \W$ one has $t(w) \cdot w = 0$. 

A semistrict Lie 2-algebra, on the other hand, which we will call simply a Lie 2-algebra in what follows, consists of similar data: an almost Lie algebra 
$(\V, \{ \cdot , \cdot \})$, a vector space $\W$, a map 
\begin{equation}\label{tWV}
t \colon \W \to \V\,,
\end{equation} an operation $\ast \colon \V \otimes \W \to \W$, and, last but not least, in addition a tensor $\gamma \in \Lambda^3 \V^* \otimes \W$. These data have to satisfy several compatibility conditions (cf.\ Eqs.\ (a)-(i) in Definition \ref{Lie2} below with $l_3\equiv \gamma$); but, in particular, for $\gamma=0$, the semi-strict Lie 2-algebra becomes strict, $(\V, \{ \cdot , \cdot \})$ turns into the Lie algebra $\mathbb{g}$ and $v \ast w$ agrees with the action $v \cdot w$ of this Lie algebra on the  $\mathbb{g}$-modul $\W$. A non-vanishing $\gamma$ governs the violation of essentially every structural identity of the strict Lie 2-algebra, like the Jacobi identity
for the bracket $ \{ \cdot , \cdot \}$ or the representation of this bracket by means of $\ast$. We thus call $\gamma$ the ''anomaly'' of the Lie 2-algebra; it corresponds to the 3-bracket in the $L_\infty$-language.

Now, all these data enter the definition of $F$ and $G$ in terms of the fundamental quantities $A$ and $\beta$ as follows,
\begin{equation} F = \mathrm{d}_M A + \tfrac{1}{2} \{ A \stackrel{\wedge}{,} A\} - t(B) \, , \quad G = \mathrm{d}_M B+ A \ast B - \tfrac{1}{6} \gamma(A \stackrel{\wedge}{,}A \stackrel{\wedge}{,}A) -
 A \circ F \, , \label{FG}
\end{equation}
where, at this point, $$ \circ \colon \V \otimes \V \to \W$$ is a completely unconstrained operation that we have not introduced above yet (an additional wedge-product between $A$ and $F$---but also $A$ and $B$---is understood). It completes the possible terms compatible with form degrees in the definitions of $F$ and $G$ (at least, if one does not use terms depending on Hodge duals of forms, which can be added only for specific dimensions of $M$).

If one chooses option {\bf 1.} for a functional governing the dynamics, then \emph{all} Lie 2-algebra data give a consistent gauge theory, for an arbitrary product $\circ$---in fact, the last term in \eqref{FG} can be made to disappear by an appropriate redefinition of the Lagrange multiplier field!---provided only that $\W$ carries a scalar product $\kappa_\W$ satisfying the invariance property
\begin{equation} \kappa_\W(v \ast w, w) = 0 \label{kappaW}
\end{equation}
for all $v \in \V$ and $w \in \W$, which reduces to mere $\mathbb{g}$-invariance for $\gamma=0$.

The situation changes drastically for option  {\bf 2.}, which is much more restrictive for what concerns the data defined on $\V$ and $\W$. Actually, we see that a \emph{necessary} condition of gauge invariance of the functional \eqref{gerbe} is that the $\W$-valued 3-form $G$ transforms into itself (up to some endomorphism of $\W$). 
The main result of \cite{Str} is that this behaviour of $G$ under gauge transformations holds true, iff the following two items are satisfied:
\begin{enumerate} \item 
$\V$ carries the structure of a Leibniz algebra, the product of which we denote by a bracket $[ \cdot , \cdot]$, such that 
\begin{equation}(\V, [ \cdot , \cdot], t \colon \W \to \V, \circ) \label{ELA}\end{equation}
forms what we call an enhanced Leibniz algebra (cf.\ Definition \ref{definition_enhanced_Leibniz} below for the axioms).
\item The Lie 2-algebra data $(\V,\{ \cdot , \cdot \},  t \colon \W \to \V, \ast, \gamma)$ are determined uniquely from the data of the ELA as follows: the bracket $\{ \cdot , \cdot \}$ is the anti-symmetric part of the Leibniz bracket, the operation $\ast$ is given by the formula $v \ast w := t(w) \circ v$, and, last but not least, the anomaly $\gamma$ is 
\begin{eqnarray}
\gamma(v_1,v_2,v_3)\,:=\,-v_1\circ[v_2,v_3]+[v_1,v_2]\circ v_3+v_2\circ[v_1,v_3] \nonumber \\
-t(v_2\circ v_3)\circ v_1-t(v_1\circ v_2)\circ v_3+t(v_1\circ v_3)\circ v_2 \, .\label{gamma}
\end{eqnarray}
This statement includes, in particular, the highly non-trivial fact that every ELA induces a Lie 2-algebra by means of the above formulas.
\end{enumerate}
For full gauge invariance of the functional \eqref{gerbe}, we need, in addition, some constraining condition on the scalar products $\kappa_\V$ and $\kappa_\W$. One finds that $\kappa_\V$ should be invariant with respect to the left action of $\V$ onto itself, $\kappa_\V([v_1,v_2],v_2)= 0$ for all $v_1,v_2 \in \V$, together with the previous relation \eqref{kappaW}, which we now can rewrite in the form $\kappa_\W(t(w) \circ v,w)= 0$ for all $v \in \V$ and $w \in \W$. We call such data, an ELA together with these invariant scalar products,  needed for the consistent definition of \eqref{gerbe}, a (postive) quadratic ELA---in analogy to the quadratic Lie algebras, needed for the definition of ordinary Yang-Mills theories \eqref{YM}.

\vspace{5mm}
This almost concludes our excursion into the mathematical physics of non-abelian gerbes and we now turn to the algebraic structures that we want to analyze in this paper. An ELA is given by the data \eqref{ELA} satisfying the Properties (a), (b), (c), and (d) in Definition \ref{definition_enhanced_Leibniz} below. 

In {\it Section 1} we study the basic properties following from these axioms and define morphisms between different ELAs. While this defines the category {\bf ELeib}, there is an important subcategory {\bf sELeib} of it consisting of symmetric ELAs, i.e.\ ELAs for which the product $\circ$ is symmetric, $v_1 \circ v_2 = v_2 \circ v_1$, together with those morphisms which preserve this symmetry. We will show an equivalence between {\bf ELeib} and {\bf sELeib}, which implies that for many purposes we may restrict to the simpler setting where the $\W$-valued product on $\V$ is symmetric. We conclude this section by specifying the hidden Lie algebra ${\mathfrak{g}}$ that governs the invariant scalar products of a quadratic ELA. 

In {\it Section 2} we address the question of how much more information one needs in addition to the underlying Leibniz algebra so as to obtain a general ELA. We first show that if the map \eqref{tWV} is injective, all the data of an sELA \eqref{ELA} correspond in a canonical way to
what we call a Leibniz couple $\V \supset {\mathfrak{i}}$, which is  the underlying Leibniz algebra $(\V, [\cdot , \cdot])$ together with the choice of an appropriate ideal ${\mathfrak{i}}\subset \V$, see Definition \ref{Leibnizcouple} and Proposition \ref{proposition_injective}. In fact, ${\mathfrak{i}}$ is singled out inside $\V$ as the image $t(\W)$ of the map \eqref{tWV}. 

A general sELA now corresponds to such a Leibniz couple together with a 2-cohomology class $[\Delta] \in H^2_{\mathrm{d}}(\V)\otimes \ker(t)$, see Theorem \ref{maintheorem}. Here every representative $\Delta$ corresponds to a map $S^2(\V^*)\to \ker(t)$ which has to vanish upon restriction to $S^2{t(\W)}$; the coboundary operator $\mathrm{d}$ turns out to not coincide with the Loday differential $\mathrm{d}_L$, but instead with a differential in another complex induced by the Loday one  (see Definition \ref{defd} and Remark \ref{remd}, item 4)). 
A $\mathrm{d}_L$ 2-class still comes into the game when one describes the Leibniz couple in terms of a Lie algebra ${\mathfrak{g}}:= \V/t(\W)$ and its module $t(\W)$. Several examples of ELAs and sELAs illustrate the findings, some of which are deferred into {\it Appendix A}. Section 2 concludes with a characterization of a positive quadratic ELA, which is what one needs for defining the action functional \eqref{gerbe}, in terms of purely Lie algebraic data, see Proposition \ref{character}.

In {\it Section 3} we construct the functor ${\cal F}$ from {\bf ELeib}  to {\bf Lie-2-alg}. The main formulas for the Lie 2-algebra induced by an enhanced Leibniz algebra \eqref{ELA} were already given in 2.\ above. This Lie 2-algebra can be seen also equivalently as a 2-term $L_\infty$-algebra defined over the complex \eqref{tWV}, with $t$ the 1-bracket, complemented by 2- and 3-brackets. These have to satisfy several higher Jacobi-type identities. The proof is rather involved and  in \cite{Str} only the result was announced. A huge technical simplification arises from the fact that one can restrict to an sELA in a first step. Given a non-symmetric ELA \eqref{ELA}, the corresponding sELA is obtained by mere symmetrization of $\circ$. The original ELA can be reobtained for a unique morphism, which we call a $\beta$-transform $\tau_{\beta}$ and which is determined by some $\beta \in \Lambda^2 \V^* \otimes \W$. One then proves the validity of the axioms of the induced Lie 2-algebra in the simplified setting of the sELA. Applying ${\cal F}(\tau_\beta)$ provides the Lie 2-algebra for the original ELA and extends the proof to this setting as well.  That the  induced Lie 2-algebra depends critically on the antisymmetric part of $\circ$-product is visible, e.g., from Equation \eqref{gamma} for the anomaly $\gamma$ in terms of the ELA data, where the order of terms left and right of $\circ$ is essential for it to satsify the higher Jacobi identities. 

For a final time we return to the original field-theoretic setting, as it sheds additional light on the last-mentioned facts. Consider the following field redefinition:
\begin{equation}B \mapsto B - \tfrac{1}{2}\beta(A \stackrel{\wedge}{,} A)\, . \label{Btransformed}
\end{equation}
This implies in particular that $\mathrm{d}_M B \mapsto \mathrm{d}_M B + \beta(A \stackrel{\wedge}{,} \mathrm{d}_M A)$. Let $\beta \in \Lambda^2 \V^* \otimes \W$ be chosen to coincide with the antisymmetric part of $\circ$. Now plugging the transformation \eqref{Btransformed} into the defining relations \eqref{FG} for $F$ and $G$, we see that all the other structural entities entering these expressions change by the redefinition as well  (except for the map $t$). This provides the formulas for ${\cal F}(\tau_\beta)$. The equivalence of 
categories of general and symmetric ELAs corresponds to a field redefinition in the gauge theory. And this same field redefinition induces also the correct formulas for the corresponding changes of induced Lie 2-algebra, so as for the functor ${\cal F}$ to be well-defined.

\vspace{.5cm}

\noindent{\bf Acknowledgements:} FW thanks Institut Camille Jordan at
Universit\'e Claude Bernard Lyon 1 for several visits during which part of this work
has been done.

\section{Enhanced Leibniz algebras and \\the symmetrization functor}

All vector spaces are over a field $\K$ of characteristic different from $2$. 

\begin{defi}   \label{definition_enhanced_Leibniz}
An enhanced Leibniz algebra (ELA) is a triple \\$(t\colon\W\to\V,[\cdot,\cdot],\circ)$ consisting of 
a linear map
$$t \colon {\mathbb W}\to{\mathbb V}$$ where $(\V, [\cdot,\cdot])$ is a left Leibniz algebra, i.e. for all $v_1,v_2,v_3\in{\mathbb V}$ one has
\begin{equation}[v_1,[v_2,v_3]]\,=\,[[v_1,v_2],v_3]+[v_2,[v_1,v_3]],\label{Leibniz}
\end{equation}
 and a (bilinear) product $\circ\colon{\mathbb V}\times{\mathbb V}\to
{\mathbb W}$ such that for all $v\in{\mathbb V}$ and all $w\in{\mathbb W}$
\begin{enumerate}
\item[(a)] $[t(w),v]\,=\,0$,
\item[(b)] $t(w)\circ t(w)\,=\,0$,
\item[(c)] $u \stackrel{s}{\circ} [v,v]\,=\,v\stackrel{s}{\circ}[u,v]$,
\item[(d)] $[v,v]\,=\,t(v\circ v)$.
\end{enumerate}
\end{defi}

By $\stackrel{s}{\circ}$ we denote the symmetrization of the $\circ$-product, $u\stackrel{s}{\circ}v\,\equiv\,\frac{1}{2}(u\circ v + v\circ u).$

\begin{defi} \label{definition_senhanced_Leibniz}
A symmetric enhanced Leibniz algebra (sELA) is an ELA $(t\colon\W\to\V,[\cdot,\cdot],\circ)$ such that $\circ$ is symmetric, i.e. $\stackrel{s}{\circ}= \circ$. 
\end{defi}

A Leibniz algebra ${\V}$ has two canonical (two-sided, abelian) 
ideals, namely the {\it ideal of squares}
$$\Sq({\V})\,:=\,\{x\in{\V}\,|\,
x=\sum\lambda_i[x_i,x_i]\},$$
and the {\it left-center} of ${\V}$
$$Z_L({\V})\,:=\,\{x\in{\V}\,|\,
\,\forall\, y\in{\V}:\,\,\,[x,y]=0\}.$$
It follows directly from \eqref{Leibniz} with $v_1=v_2$ that $\Sq({\V})\subset Z_L({\V})$. Evidently, the corresponding quotient Leibniz algebras are Lie algebras. 
Observe that $\Sq({\V})$ behaves well under morphisms
of Leibniz algebras, while this is not necessarily the case for $Z_L({\V})$. 

For later use it will be convenient to determine the polarization of the 
identities appearing in Definition \ref{definition_enhanced_Leibniz}:
\begin{lem} \label{polarization_lemma}
In an ELA $(t\colon\W\to\V,[\cdot,\cdot],\circ)$, we have the following identities
for all $w_1,w_2\in{\mathbb W}$ and all $u,v_1,v_2\in{\mathbb V}$:
\begin{enumerate}
\item[(a)] $t(w_1)\circ t(w_2)\,=\,-t(w_2)\circ t(w_1)$,
\item[(b)] $u\stackrel{s}{\circ} \left([v_1,v_2]+ [v_2,v_1]\right)\,=\,v_1\stackrel{s}{\circ}[u,v_2]+v_2\stackrel{s}{\circ}[u,v_1]$,
\item[(c)] $[v_1,v_2]+[v_2,v_1]\,=\,t(v_1\circ v_2+ v_2\circ v_1)$.
\end{enumerate}
\end{lem}
{}From each of these relations we draw a consequence that will be useful in the following (using Property (a) of an ELA for the last two of them):
\begin{cor} \label{corlemma}
For every $u,v\in \V$ and $w,w' \in \W$ one has
\begin{enumerate} 
\item[(1)]  $t(w)\stackrel{s}{\circ}t(w')\,=\,0$,
\item[(2)] $u \stackrel{s}\circ [v,t(w)] = v \stackrel{s}\circ [u,t(w)] + t(w) \stackrel{s}\circ [u,v]$,
\item[(3)] $[v,t(w)]=2t(v \stackrel{s}\circ t(w))$.
\end{enumerate}
\end{cor}
\begin{lem} \label{im_t_abelian_ideal}
In an ELA $(t\colon\W\to\V,[\cdot,\cdot],\circ)$, the image $\im(t)$ is an abelian ideal
in the Leibniz algebra $(\V,[\cdot,\cdot])$. 
\end{lem}
\begin{prf}
On the one hand, $[t(w),v]=0$ for all $w\in\W$ and all $v\in\V$ by the defining Property (a). 
On the other hand, item 3) of the above corollary implies  $[v,t(w)]\in\im(t)$. Thus  $\im(t)$ is an ideal. It is abelian, $[t(w),t(w')]=0$. 
\end{prf}

Let us describe what Property (c) means in the special case where the Leibniz algebra
$(\V,[\cdot,\cdot])$ is a Lie algebra, i.e. $[v,v]=0$ for all $v\in\V$, and the circle product $\circ$ is symmetric. 

\begin{lem} \label{Liebil}
Consider an sELA $(t\colon\W\to\V,[\cdot,\cdot],\circ)$ where $(\V,[\cdot,\cdot])$ is a Lie algebra. Then $\circ\colon\V\otimes\V\to\W$ is a symmetric invariant bilinear form with values in $\W$.
\end{lem}

\begin{prf}
Indeed, Property (c) reads in this framework $v\circ[u,v]=0$.  
By polarisation, one obtains $v\circ[v',v'']=-v''\circ[v',v]$. Together with the antisymmetry of the Lie-bracket $[\cdot,\cdot]$, this implies total antisymmetry in the three arguments of 
$$\alpha(v,v',v''):=v\circ[v',v''].$$
This means that the ($\W$-valued) symmetric bilinear form $(v,v'):=v\circ v'$
is invariant under the Lie algebra $(\V,[\cdot,\cdot])$ (acting trivially on $\W$):
$$
(v,[v',v''])
=([v,v'],v'').$$
\end{prf}
Conversely, for every Lie algebra $(\V,[\cdot,\cdot])$ together with an invariant bilinear form 
$\circ\colon\V\otimes\V\to\K$, the triple $(t\colon\W\to\V,[\cdot,\cdot],\circ)=(0\colon\K\to\V,[\cdot,\cdot],\circ)$ is an enhanced Leibniz algebra.   

More generally, returning to not-necessarily Lie-type Leibniz algebras, one has the following result.
\begin{prop}\label{Str}\cite{Str} Every Leibniz algebra $(\V,[\cdot,\cdot])$ is an sELA in a canonical way by setting $\W := \Sq(\V)$, $t$ the canonical inclusion into $\V$, and $\circ$ the symmetrization of the Leibniz bracket.
\end{prop}
\begin{prf} Properties (a), (b), and (d) are satisfied by construction and the properties of $\Sq(\V)$. It remains to verify Property (c). Putting $v_1=u$, $v_2=v_3=v$ in \eqref{Leibniz}, we obtain $[u,[v,v]]= [[u,v],v]+ [v,[u,v]]$; adding to the left hand side the term $[[v,v],u]=0$, we obtain
$$ [u,[v,v]]+ [[v,v],u]= [v,[u,v]]+[[u,v],v],$$
which, upon the identification of the symmetrization of the Leibniz bracket with the (then inherently symmetric) $\circ$-product, indeed turns precisely into (c).
\end{prf}
On inspection of Property (d) of an ELA, one concludes furthermore:
\begin{cor}\label{surjcirc}\cite{Str} 
Every sELA $(t\colon\W\to\V,[\cdot,\cdot],\circ \colon \V \otimes \V \to \W)$  with an injective map $t$ and a surjective map $\circ$ reduces canonically to its underlying Leibniz algebra 
$(\V,[\cdot,\cdot])$, with the structural maps as in the previous proposition.
\end{cor}
It is essentially these last two observations motivating the name of the algebraic structure given in Definition \ref{definition_enhanced_Leibniz}. 

\vspace{0.5cm}
We now turn to the question of morphisms:
We first observe that only the symmetric part of $\circ$ enters the axioms (a)-(d) of an enhanced Leibniz algebra. This implies the following lemma and motivates the subsequent definitions.
\begin{lem}Let $(t\colon\V\to\W,[\cdot,\cdot],\circ)$ be an ELA and $\beta \colon\V\wedge\V\to\W$. Define a new $\circ$-product by means of $v_1 \circ' v_2 = v_1 \circ v_2 + \beta(v_1,v_2)$ for all $v_1,v_2 \in \V$. Then $(t\colon\V\to\W,[\cdot,\cdot],\circ')$ is an ELA.
\end{lem}
\begin{defi} \label{betatrafo} For every $\beta \in \Lambda^2 \V^* \otimes \W$, we call the transformation $(t\colon\V\to\W,[\cdot,\cdot],\circ)\mapsto (t\colon\V\to\W,[\cdot,\cdot],\circ+\beta(\cdot,\cdot))$  a
{\it $\beta$-transformation} of an ELA. 
\end{defi}
\begin{defi}\label{plain}
Let $(t\colon\V\to\W,[\cdot,\cdot],\circ)$ and $(t'\colon\V'\to\W',[\cdot,\cdot]',\circ')$ be two enhanced Leibniz algebras.
A \emph{plain morphism} 
$$(\varphi,\psi)\colon(t\colon\V\to\W,[\cdot,\cdot],\circ)\to(t'\colon\V'\to\W',[\cdot,\cdot]',\circ')$$
is the data of a commutative square
\begin{equation} \label{diagram}
\xymatrix{ \W  \ar[r]^{t} \ar[d]^{\varphi} & \V \ar[d]^{\psi}  \\
             \W' \ar[r]^{t'} & \V'  }\end{equation}
such that $\psi\colon(\V,[\cdot,\cdot])\to(\V',[\cdot,\cdot]')$ is a morphism of Leibniz algebras, 
and for all $v_1,v_2\in\V$, we have
\begin{equation} \psi(v_1)\circ'\psi(v_2)\,=\,\varphi\left(v_1\circ v_2\right) . \label{neuescirc} \end{equation}
\end{defi}
\begin{defi} \label{morphismsELA}
The morphisms of an ELA are generated by plain morphisms and $\beta$-transforms. The morphisms of an sELA are only the plain morphisms. 
\end{defi}
With these morphisms, enhanced Leibniz algebras constitute a category, denoted by {\bf{ELeib}}. 
Evidently, the category {\bf sELeib} of symmetric enhanced Leibniz algebras is a subcategory of {\bf{ELeib}}. But in fact, one has more:
\begin{prop} \label{proposition_symmetrization}
The inclusion functor $F \colon \bf{sELeib} \to \bf{ELeib}$ defines an equivalence of
categories. 
\end{prop}
\begin{prf} 
Observe that for a $\beta$-transformation and a plain morphism $(\varphi,\psi)$ one has $(\varphi,\psi)\circ \beta=\beta'\circ(\varphi,\psi)$ with $\beta'=\varphi\circ\beta$. 
The functor $F$ is fully faithful, because in case the morphisms in ${\bf{ELeib}}$ between two symmetric ELAs involve $\beta$-transforms, they can be brought to the end always and then must cancel, because the resulting ELA is symmetric. It is also essentically surjective: every $A \in {\bf{ELeib}}$ is isomorphic to an $S(A)\in {\bf{sELeib}}$ by a (unique) $\beta$-transform. \end{prf}
\begin{rem} \label{remB}
One can also state this proposition by means of the symmetrization functor $S \colon {\bf{ELeib}} \to {\bf{sELeib}}$ which acts on objects according to $S((t\colon\V\to\W,[\cdot,\cdot],\circ))=(t\colon\V\to\W,[\cdot,\cdot],\stackrel{s}{\circ})$ and on morphisms by composition with the appropriate (unique) $\beta$-transforms and their inverses that correspond to the symmetrization. The functor $S$ is a
left-adjoint to $F$, rendering the subcategory {\bf{sELeib}} reflective.  
\end{rem}
Thus, in a first step, we can restrict our considerations always to symmetric ELAs  and 
 forget about  $\beta$-transforms, extending to all ELAs in a second step.  
 \begin{rem} Let $(\varphi,\psi)$ be a morphism of sELAs. Its kernel $t|_{\ker(\varphi)}\colon\ker(\varphi)\to\ker(\psi)$ and image $t|_{\im(\varphi)}\colon\im(\varphi)\to\im(\psi)$ define enhanced Leibniz subalgebras.
\end{rem}

We conclude this section with the definition of quadratic ELAs, i.e.\ those enhanced Leibniz algebras $t \colon \W \to \V$ where $\V$ and $\W$ carry appropriately invariant non-degenerate symmetric bilinear forms. For this we first need to define the action of $\V$ on $\V$ and $\W$. While its action on itself is given by left multiplication, the one on $\W$ is more intricate:
\begin{defi} \label{defVWirkung}The \emph{adjoint representation} on the ELA $t\colon \V \to \W$ is defined by means of the following formulas, where $u,v \in \V$ and $w \in \W$:
\begin{equation}
\rho_l(u) \cdot v := [u,v] \quad,  \qquad \rho_l(u) \cdot w := 2u \stackrel{s}{\circ} t(w)     \, .\label{VWirkung}
\end{equation}
\end{defi}
\begin{prop}
The above operators indeed represent the bracket on $\V$: \begin{equation}[\rho_l(u),\rho_l(v)]=\rho_l([u,v]). \label{Darstellung} \end{equation}
\end{prop}
\begin{prf} For the operation of $\rho_l(u)$ on $\V$, Equation \eqref{Darstellung} follows directly from the Leibniz identity \eqref{Leibniz}. By item (3) of Corollary \ref{corlemma}, one has \begin{equation}
\rho_l(u)(\rho_l(v)(w))\equiv 4 u \stackrel{s}{\circ} t ( v \stackrel{s}{\circ} t(w))=2u \stackrel{s}{\circ} [v,t(w)].\label{eqa}
\end{equation}
Reordering item (2) of the same corollary and multiplying all by a factor of 2, gives
\begin{equation}
 2u \stackrel{s}\circ [v,t(w)] - 2v \stackrel{s}\circ [u,t(w)] = 2[u,v] \stackrel{s}\circ t(w) \, .\label{eqb}
\end{equation}
Combining Equations \eqref{eqa}, \eqref{eqb}, and \eqref{VWirkung}, we obtain Eq.\ \eqref{Darstellung} also in this case.
\end{prf}
\begin{prop} \label{factors} For an ELA $t \colon \W \to \V$, define the Lie algebra $\mathfrak{g}:= \V/\mathrm{im}(t)$. The adjoint representation \ref{defVWirkung}  factors through to a $\mathfrak{g}$-action on $\V$ and $\W$, for which the map $t$ is equivariant. 
\end{prop}
\begin{prf} $\mathfrak{g}$ is well-defined and a Lie algebra due to Lemma \ref{im_t_abelian_ideal} and the fact that the symmetric part of the Leibniz bracket is modded out in the quotient due to Property (d) of the ELA. Due to Properties (a) and (b) (see also Corollary \ref{corlemma}, (1)), the image of the map $t$ always acts trivially on $\V$ and $\W$. The equivariance of $t$, $\rho_l(u) \cdot t(w) = t(\rho_l(u) \cdot w)$, follows directly from Corollary \ref{corlemma}, (3).
\end{prf}
Now it is straightforward to extend this action to arbitrary tensors on $\V$ and $\W$. This leads to the following:
\begin{defi} \label{defquadratic}
A \emph{quadratic enhanced Leibniz algebra} is an ELA $t \colon \W \to \V$ such that the vector spaces $\W$ and $\V$ are equipped with ad-invariant non-degenerate symmetric bilinear forms $\kappa_\W$ and $\kappa_\V$, i.e.\ for all $u,v \in \V$ and all $w\in \W$ one has
\begin{equation} \kappa_\V([u,v], v)=0\,, \qquad \qquad \kappa_\W(u \stackrel{s}\circ t(w), w)=0 \, .
\label{kappaVW}
\end{equation}
A quadratic ELA is called \emph{positive} in addition, if $\kappa_\W$ and $\kappa_\V$ are scalar products. We will call a Leibniz algebra \emph{quadratic} (resp. \emph{positive}) in case the ELA $0\to{\mathfrak h}$ is quadratic (resp. positive). 
\end{defi}
\begin{rem}
\item[1)] There can be different definitions of a quadratic ELA as generalizations of quadratic Lie algebras. For example, one could in addition require  $\kappa_\V([v,u], v)=0$ for all $u,v \in \V$, which  poses an additional constraint for non-Lie Leibniz algebras. The present definition is the one that is motivated by  mathematical physics resulting from the construction of gauge theories \cite{Str}.
\item[2)] Due to Proposition \ref{factors}, the bilinear forms $\kappa_\W$ and $\kappa_\V$ are in fact $\mathfrak{g}$-invariant.
\item[3)] While the condition \eqref{kappaVW} on $\kappa_\W$ does not depend on an overall prefactor, the prefactor 2 is absolutely necessary in the second Equation \eqref{VWirkung} for the representation property \eqref{Darstellung} to hold true. In particular, in the context of the Lie 2-algebra structure induced by an ELA, we will define the operation $v \ast w := t(w) \circ v$, which, for the case of a symmetric ELA, differs from the above definition by the prefactor only. This operation, on the other hand, will turn out to be an anomalous action, the anomaly of which is related to the 3-bracket of the Lie 2-algebra. (See Section \ref{sec:Lie2} for further details).
\item[4)] Let the Leibniz algebra $\V$ be a Lie algebra. Then, as follows from Lemma \ref{Liebil},  every map $p\colon \W \to \R$  defines an ad-invariant bilinear form $b_\V$ on $\V$ by means of the formula $b_\V(u,v) := p(u\stackrel{s}{\circ} v)$. If $b_\V$ is non-degenerate (or even positive), then it can serve as the $\V$-part of the bilinear forms needed for the definition of a (in the latter case positive) quadratic ELA. Note that in this particular case, $(\V,b_\V)$ by itself is a (positive) quadratic Lie algebra (and, for a positive quadratic ELA, $(\W,\kappa_\W)$ then is a representation of this Lie algebra, equipped with the invariant scalar product $\kappa_\W$).
\end{rem}

\section{Examples and structure theorem}

We start with a simple example of a non-symmetric ELA (and its symmetrization), adapting ideas from the realm of Omni-Lie algebras, see \cite{Wei} and \cite{SheLiu}:
\begin{exa} \label{example1}
Let ${\mathfrak g}$ be a Lie algebra and $M$ be a left ${\mathfrak g}$-module. 
Then $\V:={\mathfrak g}\oplus M$ becomes a Leibniz algebra with the 
(hemisemidirect product) bracket
\begin{equation}[(x,m),(y,n)]\,:=\,([x,y],x\cdot n),\label{omni}\end{equation}
defined for all $x,y \in {\mathfrak g}$ and $m,n \in M$. 
Let $\W:=M$ and introduce a bilinear product on $\V$ with values in $\W$ 
by
$$\circ\,\colon\,\V\times\V\to\W,\,\,\,((x,m),(y,n))\mapsto (x,m)\circ(y,n)=
x\cdot n.$$
With $t\colon\W=M\to\V={\mathfrak g}\oplus M$ the inclusion map 
into the second factor, it is now readily verified that the triple
$(t,[\cdot,\cdot],\circ)$ on the vector spaces $\W$ and $\V$ defines a non-symmetric enhanced 
Leibniz algebra. Symmetrization of the above $\circ$-product leads to the corresponding symmetric ELA. 
\end{exa}
In this example, the $\circ$-product is not necessarily surjective---just add a trivial representation of non-zero dimension to $M$. Thus we cannot apply Corollary \ref{surjcirc} to reduce the above example to a Leibniz algebra directly, although the map $t$ is injective (and we can restrict to a symmetric circle products due to Proposition \ref{proposition_symmetrization}). However, we have the following
\begin{defi} \label{Leibnizcouple}Let $\V$ be a Leibniz algebra and ${\mathfrak i}$ a left-ideal  ideal obeying the inclusions 
\begin{equation}\Sq({\V})\subset{\mathfrak i}\subset Z_L({\V}).\label{inclusions} \end{equation}
We call the pair $\V \supset {\mathfrak i}$ a \emph{Leibniz couple}. 
\end{defi}
\begin{prop} \label{proposition_injective}
Every sELA $(t\colon\W\to\V,[\cdot,\cdot],\circ)$ with injective map $t$ is in one-to-one correspondence with a Leibniz couple $\V \supset {\mathfrak i}$. 
The ideal ${\mathfrak i}$  corresponds to the image of $t$ in the ELA, ${\mathfrak i}= t(\W)\cong \W$. For a general, not necessarily symmetric ELA, the one-to-one correspondence is with $\V\supset {\mathfrak i}$ and  $\beta \in \Lambda^2 \V^*
\otimes \W$.\end{prop}
\begin{prf} It follows from Property (d) and Lemma \ref{im_t_abelian_ideal} that $\im(t)$ is an ideal in $\V$ which lies between the ideal of squares 
$\Sq(\V)$ and the left center $Z_L(\V)$:
\begin{equation} \Sq(\V) \subset \im(t) \subset Z_L(\V). \label{inclusion}\end{equation}
Property (a) is \emph{equivalent} to the second inclusion, Property (d) implies the first one. We now put  ${\mathfrak i}=\im(t)$.   
  
Since $t$ is injective, Property (d) of an sELA determines the $\circ$-product uniquely: 
$v_1 \circ v_2 = \frac{1}{2} t^{-1}([v_1,v_2] + [v_2,v_1])$. The right hand side is well-defined  due to the first inclusion in \eqref{inclusion}.

Property (c) now follows from
\begin{eqnarray}
u\circ[v,v] &=& \frac{1}{2}t^{-1}\left([u,[v,v]]+[[v,v],u]\right)  \label{uvv}\\
            &=& \frac{1}{2} t^{-1}\left([u,v],v]+[v,[u,v]]\right)  = v\circ[u,v], \nonumber
\end{eqnarray}
where we have used the Leibniz identity and $\Sq(\V)\subset Z_L(\V)$. 
Property (b) follows from Property (a) or from Lemma \ref{im_t_abelian_ideal} for the above bracket. For a non-symmetric $\circ$-product one just adds $\beta(v_1,v_2)$ to the above one, which does not effect any of the defining properties since they are formulated for the symmetric part of this product only. This concludes the proof of the proposition.
\end{prf}
There are two natural choices of such an ideal ${\mathfrak i}$ in every Leibniz algebra $\V$, namely ${\mathfrak i}:=  \Sq(\V)$ and ${\mathfrak i}:=  Z_L(\V)$. It is for the first option, i.e.\ for the Leibniz couple $\V \supset \Sq(\V)$, that the above proposition reduces to Proposition \ref{Str} (for the sELA). 

In Example \ref{example1}, the map $t \colon M \to {\mathfrak g}\oplus M$ is evidently injective; Proposition \ref{proposition_injective} permits to generalize this example in the following way. Let 
$\V \supset {\mathfrak i}$ be an arbitrary Leibniz couple. Then the quotient space ${\mathfrak g} := \V / {\mathfrak i}$ inherits an antisymmetric Leibniz bracket, i.e.\ the structure of a Lie algebra, and one has the following exact sequence of Leibniz algebras:
\begin{equation} 0 \to {\mathfrak i} \to \V \to {\mathfrak g} \to 0 \, ,
\label{sequence}
\end{equation}
where the ideal ${\mathfrak i}$ is even abelian. Note that, despite not indicated by the notation, the Lie algebra ${\mathfrak g}$  depends on the choice of ${\mathfrak i}$ for a given Leibniz algebra $\V$.
So, after the choice of an ideal ${\mathfrak i}$, we may describe the Leibniz algebra $\V$ as an abelian extension of the Lie algebra  ${\mathfrak g}$ by  ${\mathfrak i}$. The Leibniz bracket induces a  ${\mathfrak g}$-action on the vector space  ${\mathfrak i}$: every element $x \in {\mathfrak g}$ corresponds to a class of elements in $\V$ which we denote by $[x] \equiv x+ {\mathfrak i}$. Let $m \in {\mathfrak i}$, then $x \cdot m := [x +  {\mathfrak i},m] = [x,m]$, where on the right hand  side $m$ is viewed as an element in $\V$ due to the corresponding embedding in \eqref{sequence}. The representation property, $x\cdot(y \cdot m) - 
x\cdot(y \cdot m) = [x,y] \cdot m$, follows directly from Equation \eqref{Leibniz}. 

After the choice of an arbitrary splitting of the exact sequence \eqref{sequence} (in terms of vector spaces), one can identify elements of $\V$ with pairs $(m,x)$ as in Example \ref{example1}. Since 
${\mathfrak i}$ is abelian and inside the left center, the Leibniz bracket of two such elements must be of the form
\begin{equation} [(x,m),(y,n)]\,:=\,([x,y],x\cdot n + \alpha(x,y)) \label{sum}
\end{equation}
for some $\alpha \in \V^* \otimes \V^* \otimes {\mathfrak i}$.  Example \ref{example1} now results from the choice $\alpha=0$ and the identification of Proposition \ref{proposition_injective}. It is now a straightforward and known calculation to see that the general bracket \eqref{sum} satisfies the Leibniz property \eqref{Leibniz}, iff $\alpha$ satisfies the following equation
\begin{equation} \label{dalpha}
x \cdot \alpha(y,z) - y \cdot \alpha(x,z)=\alpha([x,y],z) + \alpha(y,[x,z])-\alpha(x,[y,z])\, .\end{equation}

This description of a Leibniz algebras can be formulated more elegantly in cohomological terms. Before recalling the necessary notions below, we present the resulting statement:
\begin{prop} \label{splitLeib}
Every Leibniz couple $\V\supset {\mathfrak i}$ is in one-to-one correspondence with a ${\mathfrak g}$-module $M$ together with a Leibniz ${\mathfrak g}$-cohomology 2-class $[\alpha]$ taking values in the Lie module. Here the Lie algebra ${\mathfrak g}$ is the quotient of $\V$ by the ideal ${\mathfrak i}$, the module $M$ is ${\mathfrak i}$, and the ${\mathfrak g}$-action on $M$ is as defined above, and, after the choice of a splitting of \eqref{sequence} and a representative $\alpha$ of the class $[\alpha]$, the Leibniz product on $\V \cong {\mathfrak g}\oplus {\mathfrak i}$ takes the form of \eqref{sum}.
\end{prop}
\begin{cor} Every sELA with injective $t$ is in one-to-one correspondence with a Lie algebra ${\mathfrak g}$, a ${\mathfrak g}$-module $M$, and an $[\alpha]\in HL^2({\mathfrak g},M)$.
\end{cor}
Although we only need a very particular module here, let us recall the general notion of a module over a Leibniz algebra ${\mathfrak h}$: It 
is a vector space $M$ with operators
$\rho_l(x)$ and $\rho_r(x)$ for every $x\in{\mathfrak h}$ such that
we have for all $x,y\in{\mathfrak h}$
$$(LLM)\,\,\,\,\rho_l(x)\circ\rho_l(y)=\rho_l([x,y])+\rho_l(y)\circ\rho_l(x),$$
$$(LML)\,\,\,\,\rho_l(x)\circ\rho_r(y)=\rho_r(y)\circ\rho_l(x)+\rho_r([x,y]),$$
$$(MLL)\,\,\,\,\rho_r([x,y])=\rho_r(y)\circ\rho_r(x)+\rho_l(x)\circ\rho_r(y).$$
These equations are most readily motivated by regarding the bi-adjoint representation of a Leibniz algebra on itself, i.e.\ $M := {\mathfrak h}$, $\rho_l(x) y := [x,y]$, and $\rho_r(x) y := [y,x]$, in which case these three equations follow directly from the Leibniz identity \eqref{Leibniz}, interpreting it three times in a slightly different way.

Given only a vector space $M$ with operators $\rho_l(x)$ for every
$x\in{\mathfrak h}$ such that $(LLM)$ is true, we call $M$ a {\it Lie-module}.
In case $\rho_l(x)=-\rho_r(x)$ for all $x\in{\mathfrak h}$, the module
is called {\it symmetric}. In case $\rho_r(x)=0$ for all $x\in{\mathfrak h}$,
the module is called {\it antisymmetric}. A Lie module may be viewed as an
antisymmetric module by putting $\rho_r(x)=0$ for all $x\in{\mathfrak h}$. (This is the case of interest for Proposition \ref{splitLeib}, with the Leibniz algebra ${\mathfrak h}$ being the Lie algebra ${\mathfrak g}$; in fact, in this way, every module of a Lie algebra becomes a module of this same algebra looked upon as a Leibniz algebra).
It can also be viewed as a symmetric module by setting $\rho_r(x)=-\rho_l(x)$
for all $x\in{\mathfrak h}$.

Given a vector space $E$ with a (not necessarily Lie or Leibniz) bracket $[,]:E\times E\to E$ and 
a vector space $F$ with linear maps $\rho_l,\rho_r:E\to\End(F)$, one may define an operator $\mathrm{D}$ by the following formula. For every linear map $f:E^{\otimes p}\to F$ and elements $x_0,\ldots,x_p\in E$, one defines
\begin{multline}  
\mathrm{D}f(x_0,\ldots,x_p)=\sum_{i=0}^{p-1}(-1)^i\rho_l(x_i)f(x_0,\ldots,\hat{x_i},\ldots,x_p)+
\label{Loday_coboundary}\\
+(-1)^{p-1}\rho_r(x_p)f(x_0,\ldots,x_{p-1})+ \\
+\sum_{1\leq i<j\leq p}(-1)^{i+1}f(x_0,\ldots,\hat{x_i},\ldots,x_{j-1},[x_i,x_j],
x_{j+1},\ldots,x_p).
\end{multline}

The Leibniz cohomology of a Leibniz algebra ${\mathfrak h}$ with values in the
module $M$ is by definition the cohomology of the cochain spaces
$CL^p({\mathfrak h},M):=\Hom({\mathfrak h}^{\otimes p},M)$ for $p\geq 0$ with
the {\it Loday coboundary operator} $\mathrm{D}$, which is then denoted $\mathrm{d}_{L}\colon CL^p({\mathfrak h},M)\to CL^{p+1}({\mathfrak h},M)$ in this context.  

The Leibniz identity {\it and} the above Equations $(LLM)$, $(LML)$ and $(MLL)$ imply
that $\mathrm{d}_{L}^2=0$. For references, see \cite{Lod} (but Loday works with right
Leibniz algebras) and \cite{Cov} (in the left setting). We remark in parenthesis that even if the Leibniz algebra is a Lie algebra, the above cohomology theory is in an essential way different from the Chevalley-Eilenberg cohomology of that Lie algebra. 


Now we are ready to understand the validity of the above proposition: The constraint \eqref{dalpha} now is readily seen to take the form $$ \mathrm{d}_{L}  \alpha= 0 \, , $$ where the cochain $\alpha$ takes values in the antisymmetric module ${\mathfrak i}$. On the other hand, changing the splitting leading to the identification of $\V$ with ${\mathfrak g} \oplus M$, adds a term of the form $\beta([x,y])-x\cdot\beta(y)= (\mathrm{d}_{L} \beta) (x,y)$ for some 1-chain $\beta$ to $\alpha(x,y)$. 

Replacing \eqref{omni} by \eqref{sum} for an arbitrary Loday ${\mathfrak g}$-2-cocycle  $\alpha$ in Example \ref{example1}, now gives a generalization of this example and, simultaneously, provides the general sELAs that are associated canonically to a Leibniz couple. 

There is another, less evident, but purely Lie theoretic description of a Leibniz couple. Define ${\mathfrak g}$ as in \eqref{sequence}. Then one has (for the proof we refer to the cited works):
\begin{prop}\label{Intertwiner} \cite{ABRW,KS18} Every Leibniz couple $\V \supset {\mathfrak i}$ is in one-to-one correspondence with a surjective equivariant map (intertwiner) $\theta \colon \V \to {\mathfrak g}$ between a ${\mathfrak g}$-module $\V$ and the adjoint representation of ${\mathfrak g}$ on itself. 
\end{prop}
\begin{rem} \item[1)] There are two main differences to the description mentioned before in Proposition \ref{splitLeib}: First, now the Leibniz algebra is defined on the ${\mathfrak g}$-module itself---it is defined by means of $[u,v] := \theta(u) \cdot v$---and, second, at the place  of the cohomology class  $[\alpha]$, now one has the Lie algebra morphism $\theta \colon \V \to {\mathfrak g}=\V/{\mathfrak i}$ instead. 

\item[2)] The data of a not-necessarily surjective intertwiner $\theta \colon \V \to {\mathfrak g}$ were called \emph{an augmented Leibniz algebra} in \cite{ABRW} and identified with the \emph{embedding tensor} \cite{Samtlebenetal} in \cite{KS18}.
\end{rem}

\begin{exa}
Let ${\mathfrak g}$ be a Lie algebra, $M$ a ${\mathfrak g}$-representation, and $\theta \colon M \to {\mathfrak g}$ a not necessarily surjective, ${\mathfrak g}$-equivariant map. Then $t\colon\ker(\theta)\hookrightarrow M$, together with $[m,n] := \theta(m) \cdot n$ and $m \circ n := \frac{1}{2} \theta(m) \cdot n + \frac{1}{2} \theta(n) \cdot m$
defines an enhanced Leibniz algebra. Here one notices that indeed $\theta (m \circ n)= \frac{1}{2} [\theta(m),\theta(n)]_{\mathfrak g}+ \frac{1}{2} [\theta(n),\theta(m)]_{\mathfrak g}=0$. 
\end{exa}
Conversely, we can state that
\begin{cor} Every ELA, even if not symmetric and $t\colon \W \to \V$ not injective, defines a surjective intertwiner $\theta \colon \V \to {\mathfrak g}$ of ${\mathfrak g}$-modules, where ${\mathfrak g}$ is the Lie algebra obtained by quotienting the Leibniz algebra $\V$ by $t(\W)$. 
\end{cor}
We now provide an example where $t$ is not injective. It is very explicit and adapted from \cite{GruStr} p.67.
\begin{exa}\label{exaGru}
 Let $\V$ be the Leibniz algebra 
$\End(\R^2)$ with the (non-Lie) bracket 
$$[X,Y]:=P(X)Y-YP(X), $$
for matrices $X=\left(\begin{array}{cc} x_{11} & x_{12} \\ x_{21} & x_{22} \end{array}\right)$
and $Y=\left(\begin{array}{cc} y_{11} & y_{12} \\ y_{21} & y_{22} \end{array}\right)$, 
where $P$ is the projection defined by
$$P\left(\begin{array}{cc} x_{11} & x_{12} \\ x_{21} & x_{22} \end{array}\right)=
\left(\begin{array}{cc} x_{11} & 0 \\ 0 & 0 \end{array}\right).$$
An explicit calculation then yields the following Leibniz bracket
\begin{equation}  \left[\left(\begin{array}{cc} x_{11} & x_{12} \\ x_{21} & x_{22} \end{array}\right),\left(\begin{array}{cc} y_{11} & y_{12} \\ y_{21} & y_{22} \end{array}\right)\right]= \left(\begin{array}{cc} 0 & x_{11} y_{12} \\ -x_{11}y_{21} & 0\end{array}\right)\, . \label{exLeib}
\end{equation}
Note that this bracket is a special case of the bracket \cite{LodPir}, Example (1.2) (b), where one has an endomorphism $P$ of an associative algebra $A$ which satisfies for all $X,Y\in A$
$$P(P(X)Y)=P(XP(Y))=P(X)P(Y),$$
two identities, which one verifies easily to hold true in our case here.
The squares $\Sq(\V)$ are the sums of matrices of the form
$$\left(\begin{array}{cc} 0 & x_{11}x_{12} \\ -x_{11}x_{21} & 0 \end{array}\right).$$  
One defines the circle product to be 
\begin{equation}
X\circ Y:=\left(\begin{array}{c} x_{11}y_{12}+x_{12}y_{11}  \\  -x_{11}y_{21}-x_{21}y_{11} \\
x_{22}y_{22} \end{array}\right)\in\R^3=:\W. \label{excirc}\end{equation}
The linear map $t$ is defined by
\begin{equation}t\left(\begin{array}{c} a \\ b \\ c  \end{array}\right):=\frac{1}{2}\,\left(\begin{array}{cc} 0 & a \\ b & 0 \end{array}\right).\label{ext}\end{equation}
Property (a) holds, because $P(\im(t))=0$. Property (b) holds, because 
$$\left(\begin{array}{cc} 0 & a \\ b & 0 \end{array}\right)\circ\left(\begin{array}{cc} 0 & a \\ b & 0 \end{array}\right)=0.$$
By construction, we have $$
t(X\circ X)=\left(\begin{array}{cc} 0 & x_{11}x_{12} \\ -x_{11}x_{21} & 0 \end{array}\right)=[X,X], $$
which proves property (d). Property (c), finally, holds because, on the one hand,
$$\left(\begin{array}{c} x_{11}y_{11}y_{12} \\ x_{11}y_{11}y_{21} \\ 0  \end{array}\right)=\left(\begin{array}{cc} x_{11} & x_{12} \\ x_{21} & x_{22} \end{array}\right)\circ\left(\begin{array}{cc} 0 & y_{11}y_{12} \\ -y_{11}y_{21} & 0 \end{array}\right),$$
and, on the other hand,
$$\left(\begin{array}{cc} y_{11} & y_{12} \\ y_{21} & y_{22} \end{array}\right)\circ\left(\begin{array}{cc} 0 & x_{11}y_{12} \\ -x_{11}y_{21} & 0 \end{array}\right)=\left(\begin{array}{c} x_{11}y_{11}y_{12} \\ x_{11}y_{11}y_{21} \\ 0  \end{array}\right).$$
In this example, $\ker(t)$ is $1$-dimensional. 
\end{exa}

We now turn to the understanding the structure of a general (symmetric) enhanced Leibniz algebra. We first recall the following exact sequence of vector spaces
\begin{equation} 0 \to \ker (t) \to \W \stackrel{t}{\to} \V \to {\mathfrak g} \to 0 \, , \label{seq0}
\end{equation}
which we can also interpret as exact sequence of Leibniz algebras, equipping $\ker(t)$ and $\W$ with an abelian bracket and ${\mathfrak g}:=\V/\im(t)$ with its canonical Lie algebra structure;  $t$  is  a morphism of Leibniz algebras due to Lemma \ref{im_t_abelian_ideal}. Let us split this sequence into two short exact ones:
\begin{equation} 0 \to \ker (t) \to \W \to \im(t) \to 0 \,\label{seq2}
\end{equation}
and 
\begin{equation} 0 \to \im (t) \to \V \to {\mathfrak g} \to 0 \, . \label{seq1}
\end{equation}
The second sequence is completely described in terms of the Leibniz couple $\V \supset \im(t)$ and, after the choice of a splitting, completely understood in terms of Proposition \ref{splitLeib}; in fact, the action of ${\mathfrak g}$ on $\im(t)$ can be described more explicitly than for a general ideal ${\mathfrak i}$ in this case: $\im(t)$ becomes a Lie module of $\V$ by means of $v \cdot t(w) := [v,t(w)]$, but this factors through to a ${\mathfrak g}$-module since,  due to \eqref{inclusion}, $v\in \im(t) \subset Z_L(\V)\subset \V$ acts trivially.

Thus, it remains to understand what the data and the axioms of an sELA add to the information of the Leibniz couple $\V \supset \im(t)$, i.e., alternatively, in addition to the information contained in \eqref{seq1}. To understand this better, we now choose a splitting
\begin{equation} \sigma \colon \im(t) \to \W \label{sigma}
\end{equation}
of the sequence \eqref{seq2}. We now establish the following
\begin{lem} \label{lemma_sigma} 
Given a map $t \colon \W \to \V$ with $\V$ a Leibniz algebra, such that 
$\V \supset \im(t)$ forms a Leibniz couple, every choice of splitting \eqref{sigma} defines a canonical $\circ$-product on $\V$ with values in $\W$ by means of the formula
\begin{equation} u \circ_\sigma v := \frac{1}{2} \sigma ([u,v]+[v,u])\, , \label{circs}
\end{equation}
turning $(t \colon \W \to \V, [\cdot, \cdot], \circ_\sigma)$ into an sELA.
\end{lem}
\begin{proof}We first observe that \eqref{circs} is well-defined due to the first inclusion \eqref{inclusion}. It remains to check the axioms. Property (b) holds true due to the second inclusion \eqref{inclusion}, Property (d) due to the fact that $\sigma$ succeeded by $t$ is the identity map. It thus remains to check Property (c), which follows  by replacing $t^{-1}$ in \eqref{uvv} by $\sigma$.
\end{proof}
For a  given map $t \colon \W \to \V$ with $\V \supset \im(t)$  a Leibniz couple and the choice of a splitting \eqref{sigma}, all possible sELAs compatible with these can be parametrized by  the difference of the $\circ$-product from the canonical one in Equation \eqref{circs} above,
\begin{equation} \Delta(u,v) := u \circ v - u \circ_\sigma v \, .\label{Delta}
\end{equation}
Property (d) of an ELA now shows that $\Delta(u,v)$ lies in the kernel of the map $t$; thus, restricting to sELAs for a moment for simplicity, one has
\begin{equation}  \label{Deltaspace}
\Delta \in S^2 \V^* \otimes \ker(t) \, .
\end{equation}
We recall from the discussion following \eqref{seq1}, that $\ker(t) \subset \W$ is a trivial $\V$-module, $v \cdot w \equiv [v,t(w)]=0$ for all $w \in  \ker(t)$. Analysis of the constraints and ambiguities on $\Delta$ now leads to the following definition
\begin{Defi} \label{defd}Let $\V$ be a Leibniz algebra. We then define the following short complex \begin{equation} 0 \stackrel{\mathrm{d}}{\longrightarrow} \V^*   \stackrel{\mathrm{d}}{\longrightarrow} S^2\V^*  \stackrel{\mathrm{d}}{\longrightarrow} \V^* \otimes S^2\V^*  \, , \label{newcomplex}
\end{equation}
where for a 1-chain $\delta$, 
\begin{equation} \label{ddelta}\mathrm{d} \delta(u,v) := \tfrac{1}{2}\delta([u,v]) + \tfrac{1}{2}\delta([v,u])
\end{equation}
and for a 2-chain $\Delta$,
\begin{eqnarray} \mathrm{d} \Delta(v_0,v_1,v_2) \nonumber&:= &\Delta([v_0,v_1],v_2) +\Delta(v_1,[v_0,v_2])-\Delta(v_0,[v_1,v_2]) \nonumber \\ & & 
\!\!\!\!\!-\Delta(v_0,[v_2,v_1])\, . \label{newd}
\end{eqnarray} 
\end{Defi}
\begin{rem} \label{remd}
\item[1)] If we introduce the notation $[u,v]_+ := \tfrac{1}{2}[u,v] +  \tfrac{1}{2}[v,u]$ for the symmetric part of the Leibniz bracket, then $\mathrm{d} \delta(u,v) = \delta([u,v]_+)$ and $\mathrm{d} \Delta(v_0,v_1,v_2) =\Delta([v_0,v_1],v_2) +\Delta(v_1,[v_0,v_2])-2\Delta(v_0,[v_1,v_2]_+)$.\item[2)] One may verify by a direct calculation that $\mathrm{d}^2 = 0$, so that \eqref{newcomplex} indeed defines a complex and the cohomologies $H_{\mathrm{d}}^1(\V)$ and 
$H_{\mathrm{d}}^2(\V)$ are well-defined.
\item[3)] 
The first line in the definition \eqref{newd} of $\mathrm{d}$ coincides with the result of the  Loday differential \eqref{Loday_coboundary}; the last term is needed, however, since for 1-chains $\delta$, one has $\mathrm{d} \delta = {\cal S}(\mathrm{d}_L \delta)$, where ${\cal S}$ is the symmetrization operator, but $\mathrm{d}_L {\cal S} \mathrm{d}_L \neq 0$. In addition, how we have presented it here, the last term \eqref{newd} is needed to make $\mathrm{d} \Delta$ symmetric in the exchange of $v_1$ and $v_2$. 
\item[4)] The above coboundary operator has a general extension and meaning. Let $\V$ be a vector space and $\V[1]$ a copy  shifted to degree -1 and denote by $e_a$ a basis of  
$\V[1]$. Denote by $L_\bullet$  the free graded Lie algebra of this odd vector space. In degree -2, for example, it is generated by elements of the form $[e_a,e_b]=e_a \otimes e_b + e_b \otimes e_a$; the plus sign is a result of the fact that elements in $\V[1]$ are odd. In degree -3 it is generated by elements of the form $[[e_a,e_b],e_c]$. Thus $L_{-1}= \V[1]$, $L_{-2}= (S^2\V)[2]$, and $L_3$ can be identified with a $\Gamma$-shaped Young diagram, which in turn is a sub-vector space of $S^2\V \otimes \V$ (shifted to degree -3). Now let  $\V$ be a Leibniz algebra. Then according to \cite{KS18}, $L_\bullet$ is equipped canonically with a Lie infinity structure (an $L_\infty$-structure concentrated in negative degrees). Together with its 1-bracket $l_1$, $L_\bullet$ thus turns into a complex. The above coboundary operator $\mathrm{d}$ is the dual of $l_1$, $\mathrm{d}=(l_1)^*$.\footnote{T.S.\ is grateful to Alexei Kotov for verifying this together with him.}
\end{rem}
Now we are ready to formulate the main structure theorem of this paper.
\begin{theo} 
\label{maintheorem}
Every symmetric enhanced Leibniz algebra $(t \colon \W \to \V, [\cdot , \cdot ], \circ)$ is in one-to-one correspondence with a Leibniz couple $\V \supset {\mathfrak i}$ and a cohomology class $[\Delta] \in H^2_{{\mathrm d}}(\V) \otimes \U$ for some vector space $\U$ satisfying 
$\Delta\vert_{S^2{\mathfrak i}}=0$. 

These data are obtained from the given sELA by taking ${\mathfrak i}:=\im(t)$, $\U:=\ker(t)$, and defining $[\Delta]$ by means of \eqref{Delta} up to the choice of a splitting \eqref{sigma}. 

Conversely, the original sELA is isomorphic to the following one: Take $\W := \U \oplus  {\mathfrak i}$, 
equipped with the obvious maps ${\mathfrak i} \stackrel{\sigma}{\rightarrow}  \U \oplus  {\mathfrak i}  \stackrel{\mathrm{pr}}{\rightarrow}  {\mathfrak i} \subset \V$, let $t$ coincide with $\mathrm{pr}$, viewed as a map into $\V$, choose a representative cocycle $\Delta$ of the above class, and define $u \circ v := \tfrac{1}{2}\sigma([u,v]+[v,u]) + \Delta(u,v)$. 

Every general ELA differs from the above symmetrized one by the addition of $\beta(u,v)$ to the previous $u \circ v$ for a unique $\beta\in \Lambda^2 \V^* \otimes \W$.
\end{theo}

\begin{prf} We already showed above that every ELA defines a unique Leibniz couple $\V \supset \im(t)$, see Lemma \ref{im_t_abelian_ideal} and \eqref{inclusion}. In particular, this also takes care of Property (a) of an ELA completely. 

Now, for a given sELA, choose a splitting \eqref{sigma} of \eqref{seq2}. Then $\W = \ker (t) \oplus \sigma(\im(t)) \cong \ker (t) \oplus \im(t) $ and $u \circ v = u \circ_\sigma v + \Delta(u,v)$ for some $\Delta \in S^2\V^* \otimes \ker(t)$, see Lemma \ref{lemma_sigma} and Equations \eqref{Delta} and \eqref{Deltaspace}. It now remains to check only that $\Delta$ is indeed a cocycle $Z^2_{\mathrm{d}}(\V)\otimes \ker(t)$ and that changing the splitting  \eqref{sigma} changes $\Delta$ by a coboundary $B^2_{\mathrm{d}}(\V)\otimes \ker(t)$. (The algebraic constraint $\Delta_{S^2{\im(t)}}=0$ is directly seen to be equivalent to Property (b) of the ELA, taking into account that $\circ$ and $\circ_\sigma$ satisfy this equation; also Property (d) is equivalent to the requirement that   $\Delta$ must take values in the kernel of $t$, as $\circ_\sigma$ already takes care of the non-trivial left-hand side of that equation).

Considering a second splitting $\sigma'$, we see that the difference of  $u\circ_\sigma v = \sigma([u,v]_+)$  and $u\circ_{\sigma'}v=\sigma'([u,v]_+)$ is the image of a linear map $\delta$ applied to $[u,v]_+ \equiv \tfrac{1}{2}([u,v] + [v,u])$ taking values in $\ker(t)$. So, indeed, $\Delta - \Delta' = \mathrm{d} \delta$ upon inspection of \eqref{ddelta}. 

To verify that $\Delta$ must indeed be a $\mathrm{d}$-cocycle, we first observe that its image with respect to $\mathrm{d}$ is symmetric in the last two entries. Thus, it is sufficient to consider $\mathrm{d}\Delta(u,v,v) = 2\left(\Delta(v,[u,v]) - \Delta(u,[v,v])\right)$, where we used the fact that $\Delta$ is symmetric. Since $\circ_\sigma$ satisfies Property (c) of the ELA by itself, this property is seen to be equivalent to $\mathrm{d}\Delta=0$. 

This concludes the proof for the case of an sELA. The general ELA is obtained by a unique $\beta$-transform from its symmetrization, see Remark \ref{remB}, which is expressed precisely in the last sentence of the theorem. 
\end{prf}

\begin{rem}
We may use also the more Lie algebra adapted data of Proposition \ref{splitLeib} to reformulate the above Theorem,  replacing $\V \supset \im(t)$ by $({\mathfrak g}, {\mathfrak i}, [\alpha])$, with  ${\mathfrak g}:=\V / \im(t)$, ${\mathfrak i}$ any ${\mathfrak g}$-module, and $[\alpha] \in H^2_{\mathrm{d}_L}({\mathfrak g},{\mathfrak i})$. The exact sequence \eqref{seq0} then takes the form
\begin{equation} 0 \to \U \to \U \oplus {\mathfrak i} \to ({\mathfrak i} \oplus {\mathfrak g})_\alpha \to {\mathfrak g} \to 0 \, , \label{seq0'}
\end{equation}
where all the maps in the sequence are the obvious ones. The Leibniz product takes the form \begin{equation} [(m,x),(n,y)]\,:=\,(x\cdot n + \alpha(x,y),[x,y]) \label{sum'}
\end{equation}
for all $m,n \in {\mathfrak i}$ and $x,y \in {\mathfrak g}$. For some applications it is useful to also break up $\Delta \in S^2\V^* \otimes \U$ subject to the constraint $\Delta\vert_{S^2{\mathfrak i}}=0$ into its pieces $\Delta_{\mathfrak g} \in S^2{\mathfrak g}^* \otimes \U$ and  $\Delta_{mix} \in {\mathfrak g}^* \otimes {\mathfrak i}^* \otimes \U$. For an sELA the $\circ$-product is then determined uniquely by  
\begin{equation}  \label{circsplit}
(m,x) \circ (m,x) = (\Delta_{\mathfrak g}(x,x) + 2 \Delta_{mix}(x,m),x\cdot m+ \alpha(x,x)) \, .
\end{equation}
The cocycle condition with respect to the differential $\mathrm{d}$ takes the form 
\begin{equation}  \Delta_{\mathfrak g}(y,[x,y])+\Delta_{mix}([x,y],m) + \Delta_{mix}(y,x\cdot m+\alpha(x,y)) - \Delta_{mix}(x,y\cdot m+\alpha(y,y)) = 0 \label{long}
\end{equation}
for all $x,y\in {\mathfrak g}$ and $m\in {\mathfrak i}$. 
The formulas \eqref{seq0'},  \eqref{sum'}, and \eqref{circsplit} display the general sELA broken up into its smallest pieces. \emph{Every} sELA is isomorphic to the above one for an appropriate  ${\mathfrak g}$, ${\mathfrak i}$, $[\alpha]$, and $[\Delta]$, where different choices of representatives $\alpha$ and $\Delta\sim (\Delta_{\mathfrak g}, \Delta_{mix})$  only change the isomorphism used for the identification. 
\end{rem}

In Appendix A we illustrate the usefulness of these findings at Example \ref{exaGru} and possible generalisations of it.

\vspace{5mm}
We now turn to quadratic ELAs, cf.\ Definition \eqref{defquadratic}. 
\begin{lem}
Let ${\mathfrak h}$ be a quadratic Leibniz algebra and ${\mathfrak i}\subset{\mathfrak h}$ be a left ideal. Then ${\mathfrak i}^{\perp}$ is a left ideal of ${\mathfrak h}$ as well. 
\end{lem}
 
\begin{proof}
By invariance,
we have for all   $x\in{\mathfrak h}$, $y\in{\mathfrak i}^{\perp}$, and $z\in{\mathfrak i}$
$$\kappa([x,y],z)=-\kappa([x,z],y)=0.$$
Thus $[{\mathfrak h},{\mathfrak i}^{\perp}]\subset {\mathfrak i}^{\perp}$. \end{proof}

\begin{cor}
Let ${\mathfrak h}\supset{\mathfrak i}$ be a Leibniz couple with a positive quadratic Leibniz algebra ${\mathfrak h}$. Then ${\mathfrak h}$ is isomorphic to a hemisemidirect product ${\mathfrak h}\cong{\mathfrak i}\rtimes{\mathfrak i}^{\perp}$. 
\end{cor}

\begin{proof}
By positivity of the quadratic form, we have ${\mathfrak h}={\mathfrak i}\oplus {\mathfrak i}^{\perp}$ as vector spaces. The extension
$$0\to{\mathfrak i}\to{\mathfrak h}\to {\mathfrak h}/{\mathfrak i}\cong{\mathfrak i}^{\perp}\to 0$$
has a section which embeds ${\mathfrak i}^{\perp}$ as a subalgebra of ${\mathfrak h}$. Since ${\mathfrak i}\subset Z_L({\mathfrak h})$, this shows that ${\mathfrak h}$ is isomorphic to a hemisemidirect product ${\mathfrak h}={\mathfrak i}\rtimes{\mathfrak i}^{\perp}$. 
\end{proof} 

In particular, necessarily $H^2_{\mathrm{d}_L}({\mathfrak g},{\mathfrak i}) \ni [\alpha] = 0$, where 
${\mathfrak g}={\mathfrak h}/{\mathfrak i}$. 
Similarly, we have for a left ${\mathfrak h}$-Lie module $M$:

\begin{lem}
Let $M$ be a left ${\mathfrak h}$-Lie module with an ${\mathfrak h}$-invariant positive non-degenerate, symmetric bilinear form and $N\subset M$ be a submodule. Then $M$ decomposes as a direct sum
$$M=N\oplus N^{\perp}.$$
\end{lem}

\begin{proof}
By positivity of the quadratic form, we have $M=N\oplus N^{\perp}$ as vector spaces. 
By its invariance, we have for all $x\in{\mathfrak h}$, $m\in N^{\perp}$, and $n\in N$
$$\kappa(x\cdot m,n)=-\kappa(x\cdot n,m)=0,$$
which implies that $N^{\perp}$ is a submodule. 
\end{proof}
Note that in an ELA, $\ker(t)\subset \W$ is a trivial $\V$-module. Taking into account in addition our previous findings for a general ELA, we are led to the following
\begin{prop} \label{character}Every positive quadratic sELA $(t\colon\W\to \V,[\cdot, \cdot],\circ,\kappa_\V,\kappa_\W)$ gives rise to and is, reciprocally, uniquely determined by: 
\begin{enumerate}
\item a positive quadratic Lie algebra $({\mathfrak g},\kappa_{\mathfrak g})$, 
\item a ${\mathfrak g}$-module ${\mathfrak i}$ equipped with positive ${\mathfrak g}$-invariant quadratic forms $\kappa_{\mathfrak i}$ and $\kappa_{\mathfrak i}'$, 
\item a Hilbert space $(\U,\kappa_{\U})$,
\item an adinvariant map $\Delta_{\mathfrak g}\colon S^2{\mathfrak g} \to \U$,
\item and a map $\Delta_{\mathrm{mix}}\colon {\mathfrak g} \otimes {\mathfrak i} \to \U$ subject to the constraint, $\forall$ $x,y \in {\mathfrak g}$, $m \in {\mathfrak i}$,
$$ \Delta_{mix}([x,y]_{\mathfrak g},m)  = \Delta_{mix}(x,y\cdot m) - \Delta_{mix}(y,x\cdot m),$$
\end{enumerate}
such that, up to isomorphism and $\forall$ $ x,y \in {\mathfrak g}$, $m,n \in {\mathfrak i}$,
\begin{eqnarray}
(\V,\kappa_\V)&=& ({\mathfrak g}\oplus {\mathfrak i},\kappa_{\mathfrak g}\oplus \kappa_{\mathfrak i})\, , \nonumber \\
(\W,\kappa_\W)&=&(\U\oplus {\mathfrak i},\kappa_{\U}\oplus \kappa_{\mathfrak i}')\, ,  \nonumber\\ 
t\colon \U\oplus {\mathfrak i} \to {\mathfrak g}\oplus {\mathfrak i}& ,& (u,m) \mapsto (0,m)\, , \nonumber\\
{} [(x,m),(y,n)]&=&([x,y]_{\mathfrak g}, x\cdot n)  \, , \nonumber\\
 (x,m) \circ (x,m) &=& (x\cdot m,\Delta_{\mathfrak g}(x,x) + 2 \Delta_{mix}(x,m)). \nonumber
\end{eqnarray}
Every positive quadratic ELA differs from this one, up to isomorphism, by a unique and otherwise unrestricted $\beta \in \Lambda^2\V^* \otimes \W$, which provides the skew-symmetric part of the $\circ$-product.
\end{prop}

\section{The Lie-2-algebra of an ELA}
\label{sec:Lie2}

The goal of this section is to associate to each enhanced Leibniz
algebra a two term $L_{\infty}$-algebra or, equivalently, a semi-strict Lie
2-algebra in the sense of Baez-Crans \cite{BaeCra} and to show that this map is functorial. 


Recall the definition of
a Lie 2-algebra or two term $L_{\infty}$-algebra, see \cite{BaeCra}:

\begin{defi}\label{Lie2}
A 2-term $L_{\infty}$-algebra is a linear map $t\colon V_1\to V_0$ together with 
a bilinear map $l_2\colon V_i\times V_j\to V_{i+j}$ for $0\leq i+j\leq 1$, which is 
denoted as a bracket $\{\cdot,\cdot\}$, and a trilinear map 
$l_3\colon V_0\times V_0\times V_0\to V_1$ which altogether satisfy the following axioms
for every $w,x,y,z\in V_0$ and $h,k\in V_1$:
\begin{enumerate}
\item[(a)] $\{x,y\}\,=\,-\{y,x\}$,
\item[(b)] $\{x,h\}\,=\,-\{h,x\}$,
\item[(c)] $\{h,k\}\,=\,0$,
\item[(d)] $l_3(x,y,z)$ is totally antisymmetric in the arguments $x,y,z$,
\item[(e)] $t(\{x,h\})\,=\,\{x,t(h)\}$,
\item[(f)] $\{t(h),k\}\,=\,\{h,t(k)\}$,
\item[(g)] $t(l_3(x,y,z))\,=\,-\{\{x,y\},z\}+\{\{x,z\},y\}+\{x,\{y,z\}\}$,
\item[(h)] $l_3(t(h),x,y)\,=\,-\{\{x,y\},h\}+\{\{x,h\},y\}+\{x,\{y,h\}\}$,
\item[(i)] 
$\{l_3(w,x,y),z\}+\{l_3(w,y,z),x\}-\{l_3(w,x,z),y\}-\{l_3(x,y,z),w\}+$ \\
$+l_3(\{x,z\},w,y)+l_3(\{w,y\},x,z)-l_3(\{w,x\},y,z)-l_3(\{w,z\},x,y)+$  \\
$-l_3(\{x,y\},w,z)-l_3(\{y,z\},w,x)\,=\,0.$
\end{enumerate}
\end{defi}

Observe that the bracket $l_2$ splits into two components. Denoting $V_1=:\W$
and $V_0=:\V$, there is an internal operation to $\V$, $\{,\}\colon \V\times\V\to\V$, as well as a component of $l_2$ corresponding to some kind of {\it operation} of $\V$ on $\W$. Put in one sentence, a
two term $L_{\infty}$-algebra is thus a complex $t\colon \W\to\V$ with an
antisymmetric bracket $\{,\}$ on $\V$ and an operation of $\V$ on $\W$
such that $t$ is equivariant with respect to the operation (Condition (e)) and satisfies
Peiffer's second relation (Condition (f)) plus
an ``anomaly'' $\gamma:= l_3$, which is a $3$-cocycle (Condition (i)) and governs, according to (g) and (h), the default of $\{,\}$ to be a Lie bracket and the $\ast$-commutator of the 
operation $\ast$ to define a morphism for this bracket. (We have only defined a left operation of ${\mathbb V}$ on 
${\mathbb W}$, while the right operation can be taken to be the negative of 
this left operation, in accordance with the Condition (b) of Definition \ref{Lie2}). 
We use this description of the Lie 2-algebraic data in what follows. 


We now provide the definition of the functor ${\cal F}$ from {\bf ELeib} to {\bf Lie-2-alg}, starting with the relation of the objects:
\begin{defi} 
\label{functorobjects}For every ELA $(t\colon\W\to\V,[\cdot,\cdot],\circ)$, the structural data of the  Lie 2-algebra associated by means of ${\cal F}$ are defined over the same complex $t\colon\W\to\V$, thus its 1-bracket  is  $l_1 \equiv t$. In addition, for all $v,v_1,v_2,v_3\in {\mathbb V}$ and all $w\in {\mathbb W}$, we put 
\begin{enumerate}
\item[($\alpha$)] $l_2(v_1,v_2) \equiv \{v_1,v_2\}\,:=\,[v_1,v_2]-t(v_1\circ v_2)$,
\item[($\beta$)] $l_2(v,w) \equiv v\ast w\,:=\, t(w) \circ v$,
\item[($\gamma$)] 
{${}$}
\vspace{-8mm}
\begin{eqnarray*}\hspace{-5mm}
l_3(v_1,v_2,v_3) &\equiv& \gamma(v_1,v_2,v_3):=\\ 
&&-v_1\circ[v_2,v_3]+[v_1,v_2]\circ v_3+v_2\circ[v_1,v_3]+\\
&&-t(v_2\circ v_3)\circ v_1-t(v_1\circ v_2)\circ v_3+t(v_1\circ v_3)\circ v_2.
\end{eqnarray*}
\end{enumerate}
\end{defi}
\begin{rem}\label{remfunctor}For this to really define a true functor, we need to check already that the image of an ELA as defined above really satisfies all structural identities of a Lie 2-algebra. This is the main task of the present section. But before doing so, we want to define the map ${\cal F}$ also on the morphisms of the two categories, since this will be useful in the proof of the first fact.
\end{rem}

 For this purpose we now recall the notion of morphisms for 2-term
$L_{\infty}$-algebras $(t\colon\W\to\V,l_2,l_3)$. There first are the obvious ones based on a commutative diagram  \ref{diagram}, which we do not want to spell out in detail and which we also call the \emph{plain morphisms} of the Lie 2-algebra. In addition, there is also the analogue of the $\beta$-transformations of ELAs, cf.\ Definition \ref{betatrafo}. Every $\beta \in \Lambda^2 \V^* \otimes \W$ gives rise to a Lie 2-algebra morphism as follows \cite{GruStr}:
\begin{eqnarray} l_2(v_1,v_2) &\mapsto&l_2(v_1,v_2)-t(\beta(v_1,v_2)) \, , \label{l2trafo}\\
 l_2(v,w) &\mapsto&l_2(v,w)+\beta(t(w),v) \, ,\label{l2trafo'} \\
  l_3(v_1,v_2,v_3) &\mapsto&l_3(v_1,v_2,v_3)-\bar{D}\beta(v_1,v_2,v_3) \nonumber \\ &&-\beta(t(\beta(v_1,v_2)),v_3) + \mathrm{cycl.}(1,2,3) \, . \label{gammatrafo}
\end{eqnarray}
Here we used the fake coboundary operator $\bar{D} \colon \Lambda^p \V^* \otimes \W \to \Lambda^{p+1} \V^* \otimes \W$ defined by means of the generalized Cartan formula:
\begin{align} \label{genCar} \bar{D}  \omega (v_1, \ldots , v_{p+1}) = 
  &  \sum_i (-1)^{i+1} v_i \ast 
  \omega(v_1,\ldots,\widehat{v_i},\ldots,v_{p+1})  \nonumber \\
  & +\sum_{i<j} (-1)^{i+j} \omega\big( \{v_i,v_j\}, \psi_0, \ldots,\widehat{v_i},\ldots\widehat{v_j},\ldots,v_{p+1} \big)
  \, .  
\end{align}
If the bracket  $\{\cdot ,\cdot \}$ were a true Lie bracket and $\ast$ a true representation, $\bar{D}$ would be nothing but the Chevalley-Eilenberg differential of the representation. This happens here only if the anomaly $\gamma$ vanishes, however, in which case $\bar{D}$ squares to zero and a (semi-strict) Lie 2-algebra reduces to a crossed module of Lie algebras (by equipping $\W$ with the bracket $\{w_1,w_2\}_\W := t(w_1) \ast w_2$), sometimes also called a strict Lie 2-algebra \cite{Baez02}. The operator $\bar{D}$ receives the interpretation of an almost Lie algebroid exterior covariant derivative in the context of Lie 2-algebroids, where its square then gives the almost Lie-algebroid curvature, which is proportional to $\gamma$ composed with $t$ in the last argument  \cite{GruStr}. That $ \bar{D}^2 \neq 0 $ 
reflects itself in the necessity of the addition of the terms in the second line in the transformation \eqref{gammatrafo}, since the consistency condition (i) in the Definition \ref{Lie2} becomes simply
\begin{equation} \bar{D}l_3 \equiv \bar{D}\gamma = 0 \label{barDgamma}
\end{equation}
in this notation. In addition, $\bar{D}$ changes under the $\beta$-transformations \eqref{l2trafo}, \eqref{l2trafo'}. 

There now is a simple consistency check to perform:
\begin{lem} The $\beta$-transformation of an ELA, Definition  \ref{betatrafo}, induces the transformations \eqref{l2trafo}, \eqref{l2trafo'}, and \eqref{gammatrafo} upon the map given in Definition \ref{functorobjects}. \label{useful}
\end{lem}
\begin{prf} By the replacement of the $\circ$-product in Definition  \ref{betatrafo}, the formulas \eqref{l2trafo} and \eqref{l2trafo'} follow immediately from the relations ($\alpha$) and ($\beta$), respectively. The formula for the change of $l_3\equiv \gamma$ is a straightforward calculation upon using the relation ($\gamma$) in Definition \ref{functorobjects}. 
\end{prf}
This permits us to define ${\cal F}$ on the morphisms (but cf.\ also Remark \ref{remfunctor}):
\begin{defi} \label{functormorphism} The functor  ${\cal F}\colon {\bf ELeib} \to$ {\bf Lie-2-alg} maps the plain morphisms and the $\beta$-transformations of an ELA as given by 
 Definition \ref{morphismsELA} to the plain morphisms and $\beta$-transformations of the corresponding Lie 2-algebras in the image of ${\cal F}$ as specified above.
\end{defi} 
 


Now we are ready to formulate the second main result of this paper:

\begin{theo}   \label{theorem_associated_L_infty_algebra}
 Definitions \ref{functorobjects} and \ref{functormorphism} define a functor from the category of enhanced Leibniz algebras to the category of (semi-strict) Lie 2-algebras. 
\end{theo}
\begin{prf} 
Since the morphisms are already shown to consistently map to one another, cf.\ Lemma \ref{useful}, it is sufficient to check that the image of an sELA with respect to ${\cal F}$ indeed gives a Lie 2-algebra. Since the $\beta$-transformations are particular morphisms of the respective category---and, in particular, a Lie 2-algebra is mapped to another Lie 2-algebra by means of the equations 
 \eqref{l2trafo}, \eqref{l2trafo'}, and \eqref{gammatrafo} \cite{GruStr}---the result then extends to general ELAs. The proof of this fact will be accomplished by means of the Lemmas
\ref{antisy} to \ref{anomaly_lemma} below. (For those, where the restriction to sELAs does not pose a significant simplification, we performed the proof for the general ELA). 

\begin{lem} \label{antisy}
The bracket $\{ \cdot , \cdot \}$ in  ($\alpha$) of Definition \ref{functorobjects}
is antisymmetric. \end{lem}
\begin{prf} This follows directly from (c) in Lemma
\ref{polarization_lemma}. 
\end{prf}


\begin{lem} \label{equivariance_of_t}
We have the equivariance property of $t$ with respect to the operation $\ast$ of 
${\mathbb V}$ on ${\mathbb W}$ and the ``adjoint operation'' of ${\mathbb V}$ on
itself, i.e.\ for all $v\in {\mathbb V}$ and all $w\in {\mathbb W}$
$$t(v\ast w)\,=\,\{v,t(w)\}.$$
Also Condition (f), the second Peiffer relation, holds true.
\end{lem}

\begin{prf}
The first part follows from (c) in Lemma \ref{polarization_lemma}
and Property (a) of Definition \ref{definition_enhanced_Leibniz}. The second part from the choices ($\alpha$) and ($\beta$) for $\{,\}$ and $\ast$, respectively, and Property (a) in
Lemma \ref{polarization_lemma}.  
\end{prf}

The bracket ($\alpha$) on ${\mathbb V}$ is in general not a Lie bracket,
the default is given by the  anomaly $\gamma$:

\begin{lem}   \label{definition_anomaly}
For all $v_1,v_2,v_3\in{\mathbb V}$,  we have
$$\{v_1,\{v_2,v_3\}\} - \{\{v_1,v_2\},v_3\}-\{v_2,\{v_1,v_3\}\}\,=\,
t(\gamma(v_1,v_2,v_3)).$$
\end{lem} 

\begin{prf}
Indeed,
$$\{v_1,\{v_2,v_3\}\}\,=\,[v_1,[v_2,v_3]-t(v_2\circ v_3)]-
t(v_1\circ([v_2,v_3]-t(v_2\circ v_3))),$$
$$-\{\{v_1,v_2\},v_3\}\,=\,-[[v_1,v_2]-t(v_1\circ v_2),v_3]
+t(([v_1,v_2]-t(v_1\circ v_2))\circ v_3),$$
and
$$-\{v_2,\{v_1,v_3\}\}\,=\,-[v_2,[v_1,v_3]-t(v_1\circ v_3)]+
t(v_2\circ([v_1,v_3]-t(v_1\circ v_3))).$$
Taking the sum of theses three lines,
the double bracket terms cancel by the Leibniz identity for $[\cdot,\cdot]$.
We stay on the left hand side with 
\begin{eqnarray*}
LHS &:=& t(v_1\circ t(v_2\circ v_3))-t(v_1\circ[v_2,v_3])-
[v_1,t(v_2\circ v_3)]\\
&-&t(t(v_1\circ v_2)\circ v_3) +t([v_1,v_2]\circ v_3)+
[t(v_1\circ v_2),v_3] \\
&-& t(v_2\circ t(v_1\circ v_3))+t(v_2\circ[v_1,v_3])+
[v_2,t(v_1\circ v_3)].
\end{eqnarray*}
Observe that the second, fourth, fifth, and eighth term occur in the expression
of $t(\gamma(v_1,v_2,v_3))$, while $[t(v_1\circ v_2),v_3]=0$ by (a) in 
Definition \ref{definition_enhanced_Leibniz}. 
We now use antisymmetry of the bracket $\{\cdot,\cdot\}$ and  (a) in 
Definition \ref{definition_enhanced_Leibniz} to treat the other terms:
\begin{multline*}
t(v_1\circ t(v_2\circ v_3))-[v_1,t(v_2\circ v_3)]=-\{v_1,
t(v_2\circ v_3)\}\,=\,\{t(v_2\circ v_3),v_1\}=\\
=-t(t(v_2\circ v_3)\circ v_1),
\end{multline*}
and in the same way
$$-t(v_2\circ t(v_1\circ v_3))+[v_2,t(v_1\circ v_3)]\,=\,
t(t(v_1\circ v_3)\circ v_2).$$
Thus 
\begin{eqnarray*}
LHS &=& -t(v_1\circ[v_2,v_3])-t(t(v_1\circ v_2)\circ v_3) 
+t([v_1,v_2]\circ v_3)\\
&+&t(v_2\circ[v_1,v_3])-t(t(v_2\circ v_3)\circ v_1)
+t(t(v_1\circ v_3)\circ v_2)\\
&=&t(\gamma(v_1,v_2,v_3)).
\end{eqnarray*}
\end{prf}

\begin{lem}\label{antis}
The anomaly $\gamma$ is totally antisymmetric in
its arguments. 
\end{lem}

\begin{prf}
It suffices to show for all $v,v_1,v_2,v_3\in{\mathbb V}$ that
$\gamma(v,v,v_3)=0$, $\gamma(v,v_2,v)=0$, and $\gamma(v_1,v,v)=0$. Indeed,
by polarization, the three equations establish the antisymmetry of $\gamma$ with respect to 
the transpositions $(12)$, $(13)$ and $(23)$, which in turn generate the symmetric group $S_3$.

$\gamma(v,v,v_3)=0$ follows from Property (d) in Definition 
\ref{definition_enhanced_Leibniz}, $\gamma(v,v_2,v)=0$ uses Properties 
(c) and (d) in Definition \ref{definition_enhanced_Leibniz} and Property (c)
in Lemma \ref{polarization_lemma}, and $\gamma(v_1,v,v)=0$ follows again
from Properties (c) and (d) in Definition \ref{definition_enhanced_Leibniz}.
\end{prf}

The anomaly also governs the default of the operation $\ast$ to define a genuine 
``action'' of $({\mathbb V},\{ , \})$ on $\W$:

\begin{lem}  \label{action_lemma} 
Let $(t\colon\W\to\V,[\cdot,\cdot],\circ)$ be a \emph{symmetric} ELA. 
For all $v_1,v_2\in{\mathbb V}$ and all $w\in{\mathbb W}$, we then have
$$v_1\ast(v_2\ast w)-v_2\ast(v_1\ast w)-\{v_1,v_2\}\ast w\,=\,
\gamma(v_1,v_2,t(w)).$$
\end{lem}

\begin{prf}
First of all
\begin{eqnarray*}
v_1\ast(v_2\ast w)-v_2\ast(v_1\ast w)-\{v_1,v_2\}\ast w = \\
= t(t(w)\circ v_2)\circ v_1 - t(t(w)\circ v_1)\circ v_2
-t(w)\circ[v_1,v_2]+t(w)\circ t(v_1\circ v_2). 
\end{eqnarray*}
Now
\begin{eqnarray*}
t(t(w)\circ v_2)\circ v_1 &=& ([t(w),v_2]-\{t(w),v_2\})\circ v_1 \\
&=& \{v_2,t(w)\}\circ v_1 \\
&=& -t(v_2\circ t(w))\circ v_1 + [v_2,t(w)]\circ v_1,
\end{eqnarray*}
and therefore also 
$$t(t(w)\circ v_1)\circ v_2\,=\,-t(v_1\circ t(w))\circ v_2 + [v_1,t(w)]\circ 
v_2.$$
Thus
\begin{eqnarray*}
v_1\ast(v_2\ast w)-v_2\ast(v_1\ast w)-\{v_1,v_2\}\ast w = \\
= -t(v_2\circ t(w))\circ v_1 + [v_2,t(w)]\circ v_1+t(v_1\circ t(w))\circ v_2 +\\
- [v_1,t(w)]\circ v_2-t(w)\circ[v_1,v_2]+t(w)\circ t(v_1\circ v_2) \\
= -t(v_2\circ t(w))\circ v_1 + [v_2,t(w)]\circ v_1+t(v_1\circ t(w))\circ v_2 +\\
- [v_1,t(w)]\circ v_2-t(w)\circ[v_1,v_2]+t(w)\circ t(v_1\circ v_2) 
\end{eqnarray*}
We now use Property (2) in Corollary \ref{corlemma} for a symmetric circle product 
to transform the sum of the second, fourth, and fifth term in the above equation:
$$[v_2,t(w)]\circ v_1- [v_1,t(w)]\circ v_2-t(w)\circ[v_1,v_2] = 0. $$
We are thus left with
\begin{eqnarray*}
v_1\ast(v_2\ast w)-v_2\ast(v_1\ast w)-\{v_1,v_2\}\ast w  = \\ 
= -t(v_2\circ t(w))\circ v_1+t(v_1\circ t(w))\circ v_2+t(w)\circ 
t(v_1\circ v_2) \, .
\end{eqnarray*}
Using again Property (2) of  Corollary \ref{corlemma} to establish
 $[v_1,v_2]\circ t(w)+v_2\circ[v_1,t(w)]-v_1\circ[v_2,t(w)]=0$, we see that the first line in the expression for $\gamma$ as given in Definition \ref{functorobjects} vanishes when $v_3$ is replaced by $t(w)$. The second line of $\gamma(v_1,v_2,t(w))$ agrees with the above expression except for one sign. This is of no relevance, however, since the appropriate expression, $t(w)\circ 
t(v_1\circ v_2)$, vanishes identically due to Property (1) of  Corollary \ref{corlemma}. \end{prf}

\begin{lem}   \label{anomaly_lemma}
Suppose again that the underlying ELA is symmetric. 
Then the anomaly $\gamma$ satisfies the cocycle identity \eqref{barDgamma}.
\end{lem}
\begin{prf}
We propose to take the Leibniz version of the above coboundary operator, i.e.\ the Loday coboundary operator $\mathrm{D}$ for the vector space $E=\V$ with the bracket $\{,\}$ and with values in the symmetric Leibniz module $F=\W$, see Equation (\ref{Loday_coboundary}):
\begin{multline}  
\mathrm{D}\gamma(v_1,v_2,v_3,v_4)=v_1\ast\gamma(v_2,v_3,v_4)-v_2\ast\gamma(v_1,v_3,v_4)+
v_3\ast\gamma(v_1,v_2,v_4)+  \\
-v_4\ast\gamma(v_1,v_2,v_3)-\gamma(\{v_1,v_2\},v_3,v_4)-
\gamma(v_2,\{v_1,v_3\},v_4)+ \\
-\gamma(v_2,v_3,\{v_1,v_4\})+\gamma(v_1,\{v_2,v_3\},v_4)+
\gamma(v_1,v_3,\{v_2,v_4\})+ \\
-\gamma(v_1,v_2,\{v_3,v_4\}); \label{coboundary_expression}
\end{multline}
this sum gives the same as the Chevalley-Eilenberg coboundary operator by the
total antisymmetry of $\gamma$. The reason for taking the Loday coboundary is
the following. 
Observe that the anomaly $\gamma$ may be seen as the Loday pseudo-coboundary of the
Leibniz $2$-cochain $f\colon v_1\otimes v_2\mapsto v_1\circ v_2=:-f(v_1,v_2)$
with values in $\W$, where $\V$ operates on 
$\W$ by the above defined operation $v\ast w:=t(w)\circ v$. Admittedly, this is
not necessarily a Leibniz action and thus $\gamma$ is in general not a Leibniz coboundary, but 
$\gamma$ has formally the correct expression:
\begin{eqnarray*}
\gamma(v_1,v_2,v_3)&=& \mathrm{D}f(v_1,v_2,v_3):= \\
&=&-f([v_1,v_2],v_3)+f(v_1,[v_2,v_3])-f(v_2,[v_1,v_3])+\\
&&v_1\ast f(v_2,v_3)-v_2\ast f(v_1,v_3)-f(v_1,v_2)\ast v_3.
\end{eqnarray*}
Here we view $\W$ as an symmetric Leibniz $\V$-pseudo-module, i.e. $w\ast v :=-v\ast w$
(although it is not a Leibniz module in general). Now one applies to this again the Loday
coboundary operator. First of all, the coboundary operator with respect to the
bracket $\{\cdot,\cdot\}$ is applied, and not the one with respect to the genuine Leibniz
bracket $[\cdot,\cdot]$. On the other hand, as the operation is not a Leibniz
action in general, we do not have $\mathrm{D}^2=0$. 
We will see that these two phenomena miraculously cancel each other. 

Indeed, one can replace the bracket $\{\cdot,\cdot\}$ by the bracket
$[\cdot,\cdot]$ in the six
bracket terms (i.e. the terms of the type  $\gamma(v_2,v_3,\{v_1,v_4\})$) of
the above expression of the coboundary (Equation (\ref{coboundary_expression}))
to the cost of six terms of the type 
\begin{multline*}
\gamma(v_2,v_3,t(v_1\circ v_4))=v_2\ast(v_3\ast(v_1\circ v_4))-v_3\ast(v_2\ast
(v_1\circ v_4))+ \\-[v_2,v_3]\ast(v_1\circ v_4)   
=:{\rm llm}(v_2,v_3)\,\,\,{\rm on}\,\,\,(v_1\circ v_4).
\end{multline*}
Here we have used Lemma \ref{action_lemma} and the fact that we may replace
the bracket $\{,\}$ in $\{v_2,v_3\}\ast(v_1\circ v_4)$ by the genuine Leibniz
bracket $[\cdot,\cdot]$ thanks to Property (b).

Replacing in this way all the brackets $\{\cdot,\cdot\}$ by Leibniz brackets
$[\cdot,\cdot]$ in the bracket terms of the coboundary expression, we obtain
on the one hand the expression $\mathrm{D}^2\,f(v_1,v_2,v_3,v_4)$ for the two consecutive
Loday coboundary operators applied to the cochain $f$, and on the other hand
the following six terms
\begin{multline}
{\rm llm}(v_1,v_2)\,\,\,{\rm on}\,\,\,(v_3\circ v_4)
-{\rm llm}(v_1,v_3)\,\,\,{\rm on}\,\,\,(v_2\circ v_4)
+{\rm llm}(v_1,v_4)\,\,\,{\rm on}\,\,\,(v_2\circ v_3) \\
+{\rm llm}(v_2,v_3)\,\,\,{\rm on}\,\,\,(v_1\circ v_4) 
-{\rm llm}(v_2,v_4)\,\,\,{\rm on}\,\,\,(v_1\circ v_3)
+{\rm llm}(v_3,v_4)\,\,\,{\rm on}\,\,\,(v_1\circ v_2). \label{*}
\end{multline}
On the other hand, it is a long but straightforward computation using the
Leibniz identity (and not the action identites which may not be satisfied !)
to verify that
\begin{eqnarray}
\mathrm{D}^2\,f(v_1,v_2,v_3,v_4)={\rm llm}(v_1,v_2)\,\,\,{\rm on}\,\,\,f(v_3,v_4) -{\rm llm}(v_1,v_3)\,\,\,{\rm on}\,\,\,f(v_2,v_4)  \nonumber \\
+{\rm llm}(v_1,v_4)\,\,\,{\rm on}\,\,\,f(v_2,v_3)
+{\rm llm}(v_2,v_3)\,\,\,{\rm on}\,\,\,f(v_1,v_4) \nonumber \\
-{\rm llm}(v_2,v_4)\,\,\,{\rm on}\,\,\,f(v_1,v_3)+
{\rm llm}(v_3,v_4)\,\,\,{\rm on}\,\,\,f(v_1,v_2). \label{**}
\end{eqnarray}
Now these two sets of six terms given in Equations (\ref{*}) and (\ref{**})
cancel each other, observing that
$f(v,v')=-v\circ v'$. This shows the cocycle property for the anomaly $\gamma$. 
\end{prf}

This concludes the proof of Theorem \ref{theorem_associated_L_infty_algebra}. \end{prf}

We conclude this section with three remarks:
\begin{rem}
It is interesting to observe which of the axioms of an ELA are strictly equivalent to which axiom of 
 a two-step 
$L_{\infty}$-algebra upon usage of the association given in Definition \ref{functorobjects}. We give a short overview of some of these relations.
\begin{itemize}
\item Property (b) of an ELA is equivalent to $t(w)\ast w'+t(w')\ast w=0$ for all $w,w'\in\W$. 
\item Property (d) of an ELA is equivalent to the antisymmetry of $\{\cdot,\cdot\}$. 
\item Granted Property (d), Property (a) of an ELA is equivalent to $t(v\ast w)=\{v,t(w)\}$ for all $v\in\V$ and all $w\in\W$. 
\item Granted Properties (a) and (d), the Leibniz identity for $[\cdot,\cdot]$ is equivalent to the property in Lemma \ref{definition_anomaly}. 
\item Granted Property (d), Property (c) is equivalent to there antisymmetry of the anomaly $\gamma$. 
\item Properties (a), (b), (c) and (d) imply the property of Lemma \ref{action_lemma}. 
\end{itemize}
The cocycle property for $\gamma$ is somehow special and does not involve a new identity of the ELA. Observe that given an $L_{\infty}$-algebra $t\colon\W\to\V$ such that for all $v,v'\in\V$ and all $w\in\W$, we have $\{v,v'\}=[v,v']-t(v\circ v')$ and $v\ast w=t(w)\circ v$, by the above we have Properties (a), (b), (c) and (d), i.e. we have an ELA. 
\end{rem}  

\begin{rem}
The object function of the functor ${\cal F}$ defined above from ELAs to $L_{\infty}$-algebras is \emph{not injective}. The action  $\ast$ determines the circle product $\circ$ only on $t(\W)\otimes\V$, 
but not necessarily on all of $\V\otimes\V$. This is an example: 

Consider the Leibniz algebra $\V_1={\rm span}_k\{x,y\}$ with brackets $[x,x]_1=y$ and all other brackets zero. Extend this to an ELA, which we call $ELA_1$ in what follows, by  putting $\W_1={\rm span}_k\{y\}$, $t_1\colon \W_1\to\V_1$ the inclusion, and the only non-trivial $\circ$-product is $x\circ_1x=y$. On the other hand, we take the zero ELA on the same inclusion $t_2\colon \W_2:={\rm span}_k\{y\}\to{\rm span}_k\{x,y\}=:\V_2$, i.e.\ the bracket $[\cdot ,\cdot ]_2=0$ on $\V_2$ and the $\circ$-product $\circ_2=0$ and call this $ ELA_2$.

Now, using Definition \ref{functorobjects}, one easily verifies that the corresponding two-term $L_{\infty}$ algebras are both the zero  $L_{\infty}$-algebra on the inclusion $t_1=t_2$, and thus in particular
$$ {\cal F}(ELA_1) = {\cal F}(ELA_2) \, ,$$
but the two ELAs are not even isomorphic, $ELA_1 \not \cong ELA_2$. 

This last observation implies, in addition, that the functor ${\cal F}$ is \emph{not full} (not surjective on the morphisms): The identity morphism on the right-hand side of 
\begin{equation} {\cal F}_{ELA_1,ELA_2}\colon \mathrm{Hom}(ELA_1,ELA_2) \to \mathrm{Hom}({\cal F}(ELA_1),{\cal F}(ELA_2)) \, ,  \label{Hommap}
\end{equation}
which is induced by any functor ${\cal F}$, has no preimage on the left-hand side.

The object function of the functor ${\cal F}$ is also \emph{not surjective}. An example of a Lie 2-algebra not in the image of the functor is the following one: Let $\V$ be a Lie algebra ${\mathfrak g}$ and $\W$ any non-trivial ${\mathfrak g}$-module. Take $l_1=t :=0$ and $l_3$ any Chevalley-Eilenberg 3-cocycle with values in $\W$, e.g.\ $l_3:=0$. This defines a (strict) Lie 2-algebra. Since $t$ is the trivial map, the action $\ast$ of the nontrivial ${\mathfrak g}$-action on $\W$ can never be obtained by  ${\cal F}$, cf.\ formula ($\beta$) of Definition \ref{functorobjects}.

But the functor  ${\cal F}$ is  \emph{faithful},  i.e.\ the map \eqref{Hommap} is 
injective for every choice of objects $ELA_1$ and $ELA_2$ inside {\bf ELeib}. This is the case since the $\beta$-part of a morphism between two given ELAs is essentially unique: via commutation, the $\beta$-transforms can be all concentrated in one spot (cf.\ the proof of Prop.\ \ref{proposition_symmetrization}).  

 This last statement does not mean, however, that for a given morphism between two Lie 2-algebras, there do not exist two different morphisms in  {\bf ELeib} with it as an image. 
 Consider the following example: As Lie 2-algebra we take $\V$ a Lie algebra ${\mathfrak g}$, $\W$ a trivial ${\mathfrak g}$-module, $l_1=t :=0$, and $l_3=0$. 
The following ELA is in the preimage of this Lie 2-algebra with respect to the functor  ${\cal F}$: Take $\V$, $\W$, and $t$ as above, put the Leibniz bracket equal to the Lie bracket on $\V={\mathfrak g}$, and $\circ:=0$. Every non-trivial choice of a $\beta$-transformation of the form $\beta := \bar{D} \alpha $
for some (non-closed) $\alpha \in  {\mathfrak g}^* \otimes \W$ does not change the Lie 
2-algebra---cf.\ Equation \eqref{gammatrafo}  and note that here $ \bar{D}^2=0$, since in this special case it is a Chevalley-Eilenberg differential---but introduces a non-vanishing part in the $\circ$-product. Thus the identity morphism of the Lie 2-algebra has several pre-images, which, however, are morphisms between different ELAs. \end{rem}   

\begin{rem}
Recall that there is a {\it skeletalization} (see \cite{BaeCra}), a {\it strictification} (see \cite{AbbWag}) and a {\it classification} (see \cite{BaeCra}) for semi-strict $L_{\infty}$-algebras. One may ask how to transpose these procedures to ELAs, in such a way that they give back the known ones on the associated $L_{\infty}$-algebra. 
\end{rem}

\newpage

\begin{appendix}
\section{Example \ref{exaGru} and the structure theorem}In this Appendix we first reconstruct Example \ref{exaGru} in terms of the structural pieces of Theorem \ref{maintheorem}.  Subsequently we study the relevant cohomologies for generalising the given example. Here we will find in particular that there is no non-trivial  $\mathrm{d}_L$-cocyle $\alpha$ so as to twist the underlying Leibniz algebra. There will be, however, non-trivial  $\mathrm{d}$-cocyles $\Delta$, which will permit us to arrive at five isomorphism classes of ELAs for the given boundary data as input, see the corresponding classification in Proposition \ref{classification} below.

\begin{exa}\label{exaGru'} Let ${\mathfrak g}$ be the abelian Lie algebra on $\R^2$, ${\mathfrak i}:=\R^2$ equipped with the ${\mathfrak g}$-action
\begin{equation} (x,y) \cdot (a,b) := (xa,-xb) \, ,\label{cdot} \end{equation}
which is readily verified to be a representation of the abelian ${\mathfrak g}$. Then define $\U = \R$, and take $\alpha=0$, which evidently satisfies the cocycle condition, $\Delta_{mix}=0$, and $\Delta_{\mathfrak g}((x,y),(x,y))=y^2$, which also satisfies the Equation \eqref{long}, as the first term in this equation vanishes for every abelian Lie algebra ${\mathfrak g}$.

We now show that these choices reproduce Example \ref{exaGru}, up to an isomorphism that we will provide. First, the isomorphism $$\psi \colon ({\mathfrak i} \oplus {\mathfrak g})_0 \stackrel{\sim}{\longrightarrow}\End(\R^2)$$ is given by
\begin{equation}\psi((a,b),(x,y)) := \left(\begin{array}{cc} 0 & a \\ b & 0\end{array}\right) + \left(\begin{array}{cc} x & 0 \\ 0 & y\end{array}\right)\equiv\left(\begin{array}{cc} x & a \\ b & y\end{array}\right) \, . \label{exmatrix} \end{equation}
Using in addition Equation \eqref{sum'} for the action \eqref{cdot}  and the choice $\alpha=0$, we readily obtain for the Leibniz bracket \begin{equation}  \left[\left(\begin{array}{cc} x & a \\ b & y\end{array}\right),\left(\begin{array}{cc} x' & a' \\ b' & y'\end{array}\right)\right]= \left(\begin{array}{cc} 0 & xa' \\ -xb' & 0\end{array}\right)\, , \label{ex'Leibniz}
\end{equation}
in coincidence with the original definition \eqref{exLeib}. Now there is a small subtlety concerning the map $t \colon \W \to \V$: In our prototype ELA in Theorem \ref{maintheorem} the map $t_{theorem} \colon \U \oplus {\mathfrak i} \to {\mathfrak i}\oplus{\mathfrak g}$ is simply of the form $(c,(a,b)) \mapsto ((a,b),(0,0))$. To reproduce the map $t \colon \R^3 \to \End(\R^2)$ in \eqref{ext}, we see that the isomorphism between $\R \oplus \R^2$ and $\R^3$,
$$\varphi \colon \U \oplus {\mathfrak i} \stackrel{\sim}{\longrightarrow}\R^3\, ,$$
 must be of the form 
$$\varphi(c,(a,b)) := \left(\begin{array}{c} 2a \\2 b \\ c  \end{array}\right)\, .$$
It remains to verify that $\circ$ is of the correct form. In the ELA constructed out of the elementary data, we find
$$ ((a,b),(x,y)) \circ  ((a,b),(x,y)) = (xa,-xb,y^2) \, .$$
Using the vector space isomorphisms $\psi$ and $\varphi$ above, indeed, this translates into Equation \eqref{excirc}.
\end{exa}

The above example shows part of the usefulness of the structural theorem. Example  \ref{exaGru} is demystified in this way: The Leibniz product is nothing but the hemi-semi-direct product of Example \ref{example1} for an abelian Lie algebra $\R^2$ with the action \eqref{cdot} on another copy of $\R^2$. And the only non-vanishing cocycle part $\Delta_{\mathfrak g}$ is now unrestricted (since the Lie algebra is abelian). Then all the previous defining formulas of   Example  \ref{exaGru} follow straightforwardly. Moreover, more importantly, now the ELA axioms do not need to be checked anymore, they are satisfied by construction. 

Having desentangled the ''atoms'' of an ELA in Theorem \ref{maintheorem}, one may look for an exhaustive class of examples with some given features. We will now illustrate this  by constructing all non-trivial generalizations of the previous example, Example \ref{exaGru'}, keeping fixed only the following data:
 ${\mathfrak g}$, the  ${\mathfrak g}$-module ${\mathfrak i}$ together with the action \eqref{cdot}),   and $\U=\R$. 
 
\begin{exa}  ${\mathfrak g}$ and the  ${\mathfrak g}$-module ${\mathfrak i}$ are as in example \ref{exaGru'}.  We first look for the \emph{most general} Leibniz algebra structure on $\End(\R^2)$ compatible with these data. This is governed by the Leibniz 2-cohomology class $[\alpha] \in H^2_{\mathrm{d}_L}({\mathfrak g},{\mathfrak i})$. 

Since ${\mathfrak g}$ is abelian, the Loday differential is rather simple, only the action-terms remain. An exact 2-form $\alpha =  {\mathrm{d}_L} \beta$ is of the form 
 \begin{equation} \label{exexact}  {\mathrm{d}_L} \beta((x,y),(x',y')) =  \left(\begin{array}{c}\! \! \!x \beta_1(x',y')\!\! \!\\ \! \!\!-x\beta_2(x',y') \!\! \!\end{array}\right) \, ,\end{equation}
where $\beta_i$ denotes the $i$-th entry (line) of the image of the map $\beta\colon \R^2 \to \R^2$ for $i=1,2$.
On the other hand, we now need to  find all $\alpha$s that are  ${\mathrm{d}_L}$-closed, i.e.\ which satisfy Eq.\ \eqref{dalpha}: its right hand  side is zero since the Lie algebra is abelian; thus, in the present notation, this equation takes the simple form
\begin{equation} \label{exclosed} \left(\begin{array}{c}\! \! \!x \alpha_1((x',y'),(x'',y''))\!\! \!\\ \! \!\! -x\alpha_2((x',y'),(x'',y'') \!\! \!\end{array}\right) = \left(\begin{array}{c}\! \! \!x' \alpha_1((x,y),(x'',y''))\!\! \!\\ \! \!\! -x'\alpha_2((x,y),(x'',y'') \!\! \!\end{array}\right) \, .
\end{equation}
It is now not difficult to convince oneself that \emph{every} $\alpha$ satisfying Equation \eqref{exclosed} is of the form \eqref{exexact}. In other words, there is \emph{no} non-trivial twist by a cocylcle $\alpha$ of the Leibniz product \eqref{ex'Leibniz}. 

To illustrate this last point further, consider, for example, the much more involved Leibniz bracket 
\begin{equation}  \left[\left(\begin{array}{cc} x & a \\ b & y\end{array}\right),\left(\begin{array}{cc} x' & a' \\ b' & y'\end{array}\right)\right]'= \left(\begin{array}{cc} 0 & x(a'+x'+2y') \\ -x(b' +3x'+4y')& 0\end{array}\right)\, . \label{ex''Leibniz}
\end{equation}
By its construction, we do not only  know that this bracket satisfies the defining property \eqref{Leibniz}, we know in addition that this Leibniz algebra structure on $\End(\R^2)$ is isomorphic to the much simpler one given by Equation \eqref{ex'Leibniz}; we leave it as an exercise to construct this isomorphism explicitly. 

We are now left with the computation of the $\mathrm{d}$-2-class $[\Delta]$. We remark in parenthesis that also for this purpose it is comforting to know that we can rely on the simple Leibniz product \eqref{ex'Leibniz}; computations using, e.g., \eqref{ex''Leibniz} would be much more involved on an intermediary level. 

First we look at 2-coboundaries $\mathrm{d}\delta$. Note that the classes take values in $\U=\R$ only. Inspection of the product \eqref{ex'Leibniz} and the definition \eqref{ddelta} show that exact $\Delta$s evaluated on an unprimed and a primed matrix of the form \eqref{exmatrix} equal 
\begin{equation} \lambda_1 (x a' + x' a) + \lambda_2 (x b' + x' b) \, , \label{weg}
\end{equation}
for some $\lambda_1,\lambda_2 \in \R$. 

Next the cocycle condition for a 2-cochain $\Delta$. This condition, by using the original definition \eqref{newd}, is now defined on three arguments, each living on $\R^4$. It thus seems more convenient to instead look at the equivalent Equation \eqref{long}. We already remarked that, since ${\mathfrak g}$ is abelian, $\Delta_{\mathfrak g}$ is unrestricted by this equation. Moreover, it is never of the form \eqref{weg} and thus one has
\begin{equation} \label{Deltaquadratic}
\Delta_{\mathfrak g}((x,y),(x',y')) = \mu_1 x x' + \mu_2 y y' + \mu_3 (x y' + x' y)
\end{equation}
for some $\mu_1,\mu_2, \mu_3 \in \R$, as a \emph{non-trivial} generalization of the previous $yy'$. 
It remains to consider the constraint on $\Delta_{mix}$; here, making use of $\alpha=0$, Equation \eqref{long} reduces to
\begin{equation} \label{neu}  \Delta_{mix}((x',y'),(xa,-xb))  = \Delta_{mix}((x,y),(x'a,-x'b))  \,
\end{equation}
for all $x,y,x',y',a,b \in \R$. It is now an elementary exercise in linear algebra to verify that the solutions to \eqref{neu} are always of the form \eqref{weg}. Thus, without loss of generality and up to isomorphism, we have $\Delta=\Delta_{\mathfrak g}$. We display the  result for the new $\circ$-product, which we denote by $\circ'$, on equal arguments for simplicity:
\begin{equation}  \left(\begin{array}{cc} x & a \\ b & y\end{array}\right) \circ' 
\left(\begin{array}{cc} x & a \\ b & y\end{array}\right)= \left(\begin{array}{c} 2xa \\ -2xb \\ \mu_1 x^2 + \mu_2 y^2 + 2\mu_3 x y\end{array}\right)\, . \label{ex'''Leibniz}
\end{equation}
The deformation space of  Example \ref{exaGru} or \ref{exaGru'} as an sELA with unchanged exact sequence \eqref{seq0} is thus at most 3-dimensional. In fact, this is an important other issue to illustrate by means of this example: While Theorem \ref{maintheorem} gives us tools for constructing  classes of examples as the above one, it is not yet giving a bijection to isomorphism classes. In the present case, this needs extra work  and leads to the following result:
\begin{prop}\label{classification}  There are precisely five isomorphism classes of (not necessarily symmtric) ELAs for which the Lie algebra ${\mathfrak g}$ in  Theorem \ref{maintheorem} is abelian two-dimensional, the ${\mathfrak g}$-module ${\mathfrak i}$ is another copy of $\R^2$ with the ${\mathfrak g}$-action given by Equation \eqref{cdot}, and where $\U = \ker(t)$ is 1-dimensional. 

Representatives of these classes are provided by the following sELAs: \\
$t \colon \R^3 \to \R^4, (a,b,c) \mapsto (a,b,0,0)$, the Leibniz algebra on $\R^4$ takes the form
$$ [(a,b,x,y),(a',b',x',y')] = (xa',-xb',0,0)$$ in all five cases and the (symmetric) $\circ$-product reads
$$ (a,b,x,y) \circ  (a,b,x,y) = (xa,-xb,\epsilon_1 x^2 + \epsilon_2 y^2 + \epsilon_3 xy) \, ,$$
where 
$$ (\epsilon_1,\epsilon_2, \epsilon_3) \in \{(1,0,0), (0,1,0), (1,1,0), (-1,1,0),(0,0,1) \} \, .$$
\end{prop} 
\begin{prf}[Sketch of the Proof] We already found above that, up to isomorphism, the Leibniz algebra structure on $\V = \End(\R^2)$ is the only abelian Leibniz extension of ${\mathfrak g}$ by ${\mathfrak i}$ here, since all cohomology classes $[\alpha]$  are necessarily trivial. In addition,  the map $t$ is rigid and we normalized it to a mere projection composed with an embedding. It is also clear that a general ELA can be brought into the form of an sELA by an isomorphism, cf.\ Proposition \ref{proposition_symmetrization}. It thus remains to show that---up to isomorphisms---one can choose the constants $\mu_1$, $\mu_2$, and $\mu_3$ in Equation \eqref{ex'''Leibniz} to take one of the five possible value combinations given above (after taking the rescaling of $t$ into account certainly, which in particular effects the first two entries on the right hand side of  the $\circ$-product). 

For this purpose, one may continue as follows: First, one determines all automorphisms $\psi$ of the Leibniz algebra on $\R^4$, and subsequently of accompanying automorphisms $\varphi$ such that the diagram \eqref{diagram} (for $\V'=\V=\R^4$, $\W'=\W=\R^3$, and $t'=t$) commutes. A somewhat lengthy calculation shows that this gives two seven-parameter famililes of such automorphisms. Now, we can apply the corresponding transformation on Equation \eqref{ex'''Leibniz}, so as to obtain the change of $\circ'$ to a possibly new or simplified product $\circ$:
\begin{equation}\label{huhu}
v_1\circ v_2 := \varphi^{-1}\left(\psi(v_1)\circ' \psi(v_2)\right) \, .
\end{equation}
This transformation behavior follows from Equation \eqref{neuescirc}.  

We already found above that the cohomology of $\Delta$ splits into one for $\Delta_{mix}$  and $\Delta_{\mathfrak g}$ and that the first one is trival. Requiring that \eqref{huhu} keeps $\Delta_{mix}=0$, reduces the parameter spaces of the two morphism families to five dimensions, two of which do not act non-trivially on the $\circ$-product. Thus there remain two 3-parameter families of transformations acting on the product by means of Equation \eqref{huhu}. Let us denote the three parameters by $A$, $B$, and $C$, where invertibility of the transformations $\varphi$ and $\psi$ imply that $A$ and $B$ are non-vanishing, while $C \in \R$ arbitrary. The action of this group of transformations  on $\Delta_{\mathfrak g}$ is given by the following formula (equal for both families of transformations):
\begin{equation}\label{Deltaneu}\Delta^{new}_{\mathfrak g}(x,y) = A \, \Delta^{old}_{\mathfrak g}(x,\tfrac{y-Cx}{B}) \, ,\end{equation}
where, for simplicity, we did not duplicate the arguments in $ \Delta_{\mathfrak g}$, working with the corresponding quadratic form in \eqref{Deltaneu}. Applying this to the expression found in Equation \eqref{Deltaquadratic} for $x=x'$, $y=y'$, a case by case study leads to the five representatives in the proposition for appropriate choices of $A$, $B$, and $C$. 
\end{prf}
\end{exa}
This concludes the illustration that the structure theorem presented in this paper may be quite useful for the construction of new examples as well as for an eventual classification of ELAs of particular types, e.g.\ for some types of Lie algebras ${\mathfrak g}$ and ${\mathfrak g}$-modules ${\mathfrak i}$. In general, the analysis will be, however, considerably more involved than in the example: recall that the cohomologies for $\alpha$ and  for $\Delta_{mix}$ were trivial, and that every $\Delta_{\mathfrak g}$ presented already a cohomology class by itself; generically, the two parts of $\Delta$ do not even decouple in the cohomological problem. And after having solved the involved cohomologies, as illustrated in the example above, one still needs to study the action of the automorphism group on what one obtained in this way.

\end{appendix}


\begin{thebibliography}{99}

\bibitem{AbbWag} Abbaspour, Hossain; Wagemann, Friedrich.
\emph{On $2$-Holonomy}. \newline
{\tt arXiv:1202.2292}.

\bibitem{ABRW} Alexandre, Charles; Bordemann, Martin; Rivi\`ere, Salim;
Wagemann, Friedrich.
\emph{Algebraic deformation quantization of Leibniz algebras.}  \newline
accepted in Commun. Algebra.

\bibitem{Baez02} Baez, John. \emph{Higher Yang-Mills theory.} \newline {\tt arXiv:hep-th/0206130}.
\bibitem{BaeCra} Baez, John; Crans, Alissa.
\emph{Higher-dimensional algebra VI: Lie 2-algebras.}  \newline
Theory Appl. Categ. {\bf 12} (2004), 492--538.

\bibitem{Cov} Covez, Simon.
\emph{The local integration of Leibniz algebras.}   \newline
Ann. Inst. Fourier (Grenoble) {\bf 63} (2013), no. 1, 1--35.

\bibitem{Samtlebenetal} de Wit, Bernard; Samtleben, Henning; Trigiante, Mario. 
\emph{On Lagrangians and gaugings of maximal supergravities.}\newline Nucl. Phys. {\bf B 655}(2003) 93--126.
 
\bibitem{GruStr} Gr\"utzmann, Melchior; Strobl, Thomas.
\emph{General Yang-Mills type gauge theories for p-form gauge fields: 
from physics-based ideas to a mathematical framework or from Bianchi 
identities to twisted Courant algebroids.}   \newline
Int. J. Geom. Methods Mod. Phys. {\bf 12}, no. 1 (2015)  1550009.

\bibitem{KS18} Kotov, Alexei; Strobl, Thomas. \emph{The Embedding Tensor, Leibniz-Loday Algebras, and Their Higher Gauge Theories.}
 \newline {\tt arXiv:hep-th/1812.08611}.

\bibitem{SheLiu} Liu, Zhangju; Sheng, Yunhe.
\emph{From Leibniz algebras to Lie 2-algebras.}   \newline
Algebr. Represent. Theory {\bf 19} (2016), no. 1, 1--5.

\bibitem{Lod} Loday, Jean-Louis.
\emph{Une version non commutative des alg\`ebres de Lie: les alg\`ebres de
Leibniz.}   \newline
Enseign. Math. (2) {\bf 39} (1993), no. 3-4, 269--293.

\bibitem{LodPir} Loday, Jean-Louis; Pirashvili, Teimuraz.
\emph{Universal enveloping algebra of Leibniz algebras and (co)homology}.   \newline
Math. Ann. {\bf 296} (1993) 139--158.

\bibitem{Str} Strobl, Thomas.
\emph{Non-abelian Gerbes and Enhanced Leibniz Algebras.}  \newline
Phys. Rev. D {\bf 94} (2016), no. 2, 021702.
  
\bibitem{Wei} Weinstein, Alan.
\emph{Omni-Lie algebras.}   \newline
Microlocal analysis of the Schr\"odinger equation and related topics 
(Kyoto, 1999).
S\"urikaisekikenky\={u}sho K\={o}ky\={u}roku No. {\bf 1176} (2000), 95--102.

\end{thebibliography}
\end{document}